\documentclass{imanum}

\usepackage{graphicx}
\usepackage{pbox}
\usepackage{float} 
\usepackage{bm}
\usepackage{enumitem}
\usepackage{subfigure}
\usepackage{dsfont}
\usepackage{nicefrac}
\usepackage[labelfont=bf, justification=justified]{caption}
\usepackage{color}
\usepackage{pgfplots}
\pgfplotsset{compat=newest} 
\usepgfplotslibrary{groupplots}

\usepackage{pgfplotstable}    
\usepackage[]{externaltikz}
\usepackage{relinput}
\DeclareMathOperator{\R}{\mathbb{R}}
\DeclareMathOperator{\N}{\mathbb{N}}
\DeclareMathOperator{\p}{\mathbb{P}}
\DeclareMathOperator{\domega}{\mathrm{d}\mathbb{P}(\omega)}
\DeclareMathOperator{\E}{\mathbb{E}}

\DeclareMathOperator{\G}{\mathbf{G}}

\DeclareMathOperator{\fu}{\mathbf{u}}
\DeclareMathOperator{\fv}{\mathbf{v}}
\DeclareMathOperator{\fw}{\mathbf{w}}
\DeclareMathOperator{\fR}{\mathcal{\mathbf{R}}}
\DeclareMathOperator{\cR}{\mathcal{R}}
\DeclareMathOperator{\cE}{\mathcal{E}}
\DeclareMathOperator{\fS}{\mathbf{S}}
\DeclareMathOperator{\ff}{\mathbf{f}}
\DeclareMathOperator{\fpsi}{\bm{\psi}}

\DeclareMathOperator{\dbracei}{[\![}
\DeclareMathOperator{\dbraceo}{]\!]}
\DeclareMathOperator{\Var}{\text{Var}}
\DeclareMathOperator{\reconst}{\hat{\fu}^{st}}
\renewcommand{\P}{\mathbb{P}}

\newlength{\figureheight}
\newlength{\figurewidth}
\newcounter{tikzsubfigcounter}[figure]
\renewcommand{\thetikzsubfigcounter}{\the\numexpr\value{figure}+1\relax\alph{tikzsubfigcounter}}

\newcounter{tikzsubfigcounterinvisible}[figure]
\renewcommand{\thetikzsubfigcounterinvisible}{\the\numexpr\value{figure}+1\relax\alph{tikzsubfigcounterinvisible}}

\begin{document}

\title{A posteriori error
analysis for random scalar conservation laws using the Stochastic Galerkin method}
\shorttitle{A posteriori error analysis for random scalar conservation laws}

\author{%
{\sc
Fabian Meyer\thanks{Email: fabian.meyer@mathematik.uni-stuttgart.de},
Christian Rohde}\\[2pt]
Institute of Applied Analysis and Numerical Simulation, University of Stuttgart, Pfaffenwaldring 57, 70569 Stuttgart, Germany \\[6pt]
{\sc and}\\[6pt]
{\sc Jan Giesselmann}\\[2pt]
MathCCES, RWTH Aachen University, Schinkelstra\ss e 2, 52062 Aachen, Germany
}

\shortauthorlist{F. Meyer \emph{et al.}}

\maketitle


\begin{abstract}{
In this article we present an a posteriori error estimator for the spatial-stochastic error of a Galerkin-type discretization of an initial value problem for a random hyperbolic conservation law.
For the stochastic discretization we use the Stochastic Galerkin method and for the spatial-temporal discretization of the Stochastic Galerkin system a Runge-Kutta Discontinuous Galerkin method. 
The estimator is obtained using smooth reconstructions of the discrete solution. 
Combined with the relative entropy stability framework of \cite{dafermos2005hyperbolic}, 
this leads to computable error bounds for the space-stochastic discretization error.

Moreover, it turns out that the error estimator  admits a splitting into one part 
representing the spatial error, and a remaining term, which can be interpreted as the stochastic error.
This decomposition allows us to balance the errors arising from spatial and stochastic discretization.
We conclude with some numerical examples confirming the theoretical findings.}
{hyperbolic conservation laws, random pdes, a posteriori error estimates, stochastic Galerkin method, discontinuous Galerkin method, relative entropy}
\end{abstract}

\maketitle

\section{Introduction}
In numerical simulations, accounting for uncertainties in input quantities has become an important issue in the last decades. 
The two main sources for uncertainties are the limitations in measuring  physical parameters exactly and the absence of knowledge of the underlying physical processes. 
Therefore, a whole new field of research, Uncertainty Quantification, has evolved in recent years. 
It addresses the question on how much we can rely on highly accurate numerical solutions if there are uncertain parameters. 

For the quantification of uncertainties, there exist two major approaches. 
On the one hand statistical approaches such as Monte-Carlo type methods sample the random space to obtain statistical information,
like mean, variance or higher order moments of the corresponding random field. On the other hand non-statistical approaches, 
like the intrusive and non-intrusive polynomial chaos expansion, approximate the random field 
by a series of polynomials and derive deterministic models for the stochastic modes. 

The theoretical foundation for the polynomial chaos expansion has been laid in \cite{Wiener} 
and can be described as a  polynomial approximation of Gaussian random variables to represent random processes. 
Later, the approach has been generalized to a broader class of distributions. The authors in \cite{XiuKarniadakis} considered polynomials from the so-called Askey-scheme,
where the approximating polynomials are orthonormal polynomials 
with respect to an inner product induced by the corresponding probability density function of the chosen probability distribution.
The intrusive spectral projection, also known as Stochastic Galerkin (SG) approach, 
considers a weak formulation of the partial differential equation with respect to the stochastic variable 
and uses the corresponding orthonormal polynomials as ansatz and test functions. 
Compared to elliptic and parabolic equations (\cite{Ghanem,20.500.11850/154978,10.2307/4101387}) the development of the SG method
for hyperbolic equations faces additional problems. From an analytical point of view  (nonlinear) elliptic and parabolic equations are better understood than (nonlinear) hyperbolic equations. Moreover, solutions of nonlinear hyperbolic equations may develop shocks (discontinuities) in finite time 
and these shocks also propagate into the stochastic space, cf. \cite{LEMAITRE200428}. Another problem is the possible loss of hyperbolicity 
of the resulting SG system (\cite{Despres2013}). Therefore new approaches to treat these problems have to be developed. 
For an overview on recent work on Uncertainty Quantification for hyperbolic equations see \cite{mishraUQ, maitre2010spectral, pettersson2015polynomial}. 
New numerical schemes for the SG system can be found in \cite{jin2017discrete,MR3269040, doi:10.1137/050627630} and for the convergence analysis of approximate  solutions of the SG system
see \cite{420bdedc3eb340c0b55bf59b120d7e5d, Xiu, EPFL-ARTICLE-185741}.
However, in all these works the dimension of the discrete stochastic
space is chosen in an ad hoc way, in particular independent of the spatial-temporal resolution.
Space-stochastic adaptive schemes, based on heuristic indicators, are treated in \cite{MR3351779, doi:10.1137/120863927}.  
A more reliable link between stochastic and space-time resolution can be provided by an appropriate a posteriori error estimator.

For the a posteriori analysis of deterministic scalar conservation laws,
rigorous approaches accompanied with corresponding adaptive algorithms can be found in \cite{gosse2000two, kroner2000posteriori},  
see also \cite{cockburn2003continuous, dedner2007error} and references therein. 
These estimates were derived by exploiting Kru\v{z}kov's estimates. A different approach for the a posteriori analysis 
uses reconstructions of the discrete solutions to estimate the error in a suitable energy norm, cf. \cite{Georgoulis2014}, or in the $L^2$-norm (\cite{DG_16, GMP_15}). 
The approach in \cite{DG_16} combines the reconstruction of the discrete solutions with the relative entropy framework \cite{dafermos2005hyperbolic}. 
In contrast to Kru\v{z}kov's scalar theory, this approach is also applicable for hyperbolic systems that are endowed with a strictly convex entropy/entropy-flux pair. 

While there exists a  rather firm theory for the a posteriori analysis of random elliptic and parabolic equations 
in combination with the SG method (\cite{MR2813239,MR1870425, MR3154028,MR2369381,MR3396480} ), the a posteriori analysis for hyperbolic problems is less developed. 
Beside the missing a posteriori analysis, most applications of the SG method for hyperbolic equations are rather ad-hoc, 
in the sense that the resolution in the stochastic space is rather arbitrary and, in particular, not related to resolution in space and time. 
We address exactly this issue and, in particular, provide a posteriori  error control in the form of a computable upper bound for the ``overall'', i.e., spatio-temporal and stochastic, error.
Moreover, 
the derived error estimator splits into two parts. One part quantifies errors caused by discretising stochastic space while the other part quantifies errors due to discretization of time and physical space. This allows a precise tuning of the stochastic discretization with respect to the space-time discretization.
 
We study uncertainty in hyperbolic conservation laws resulting from uncertain initial data, 
para-metrized by absolutely continuous random variables or random source terms.
The main contribution of this paper is an a posteriori error estimator for a class of numerical schemes.
In the schemes under consideration the random conservation law is discretised in stochastic space by the SG method. 
The resulting deterministic SG system is then numerically solved by a Runge-Kutta Discontinuous Galerkin method. 
Our a posteriori analysis follows the approach in \cite{DG_16} in that it makes use of a space-time-stochastic reconstruction of the discrete solutions. 
It turns out that this reconstruction satisfies a perturbed version of the original random conservation law with computable residual,
so that
the relative entropy framework allows us to
 prove a computable upper bound for the space-stochastic error. This upper bound is the main result of this paper (Theorem \ref{stochApost}).
The resultant estimator is valid also for discontinuous solutions, but it is only bounded as long as solutions are Lipschitz continuous.
A particular feature of our analysis is that the corresponding residual admits an orthogonal decomposition into a stochastic residual and a spatial residual (Theorem \ref{normOrtho}). 
This splitting enables us to determine if the overall error is dominated by the stochastic error, or the spatial error and thus, if we should increase resolution in stochastic space, 
or refine in physical space to efficiently decrease the overall error. 
Furthermore, our a posteriori analysis complements the a priori analysis in \cite{420bdedc3eb340c0b55bf59b120d7e5d, Xiu}, where the authors proved spectral convergence of the approximate  solutions of the SG system. In our numerical experiments the stochastic residual converges spectrally, i.e. our estimator for the SG-discretization error  converges in the same way as the SG-discretization error according to the findings of \cite{420bdedc3eb340c0b55bf59b120d7e5d, Xiu}.

The outline of this work is as follows: In Section 2,  we recall the notion of random entropy solutions and give a short introduction to the SG method. 
In Section 3 we describe how to reconstruct the numerical solutions for fully-discrete SG Discontinuous Galerkin schemes. 
Furthermore, we show how to construct so-called space-time-stochastic reconstructions for the random scalar conservation law. 
We then prove the above mentioned a posteriori error estimate and the orthogonal decomposition. 
Finally, in Section 4 we present numerical experiments illustrating the scaling behavior of the two parts of the space-stochastic residual.
\section{Notation and Preliminaries}
Let $(\Omega, \mathcal{F},\p)$ be a probability space, where $\Omega$ is the set of all elementary events $\omega \in \Omega$, $\mathcal{F}$ is a $\sigma$-algebra
on $\Omega$ and $\P$ is a probability measure. We further consider a second measurable space ($E, \mathcal{B}(E))$, with $E$ being a Banach space 
and $\mathcal{B}(E)$ the corresponding Borel $\sigma$-algebra. An $E$-valued random field  is any mapping $X:\Omega \to E$ such that 
$\{\omega \in \Omega~: ~X(\omega) \in B\}\in \mathcal{F}$ holds for any $B\in \mathcal{B}(E)$. 
For $ p \in [1,\infty) \cup \{\infty\}$ we consider the Bochner space $L^p(\Omega;E)$ of $p$-summable $E$-valued random variables $X$ equipped with the norm
\begin{align*}
\|X\|_{L^p(\Omega;E)} :=
\begin{cases}
(\int \limits_\Omega \|X(\omega)\|_E^p ~\domega)^{1/p} = \E(\|X\|_E^p)^{1/p},\quad& 1\leq p<\infty, \\
\operatorname{esssup}\limits_{\omega \in \Omega} \|X(\omega)\|_E,  &p= \infty .
\end{cases} 
\end{align*}

In the following we consider absolutely continuous, real-valued random variables $\xi:\Omega\to \R$, this means that there exists a density function 
$p_{\xi} :\R \to \R_+$, such that $\int \limits_{\R} p_{\xi}(y)~ \mathrm{d}y=1$ and  $\P[ \xi(\omega) \in A]=\int \limits_A p_{\xi}(y)~ \mathrm{d}y$,  for any $A \in \mathcal{B}(\R)$.

For $T \in (0,\infty)$ and $f\in C^2(\R)$ the equation of interest is the scalar conservation law with uncertain initial data and source term:
\begin{equation} \label{stochCL}  \tag{RIVP}
\begin{aligned} 
&\partial_t u(t,x,\xi(\omega))+ \partial_x f(u(t,x,\xi(\omega)))=S(t,x,\xi(\omega)) ,&(t,x,\omega) &\in (0,T)\times [0,1]_{per}\times \Omega,
\\& u(0,x,\xi(\omega))= u^0(x,\xi(\omega)), ~ &(x,\omega) &\in [0,1]_{per}\times \Omega. \notag
\end{aligned} 
\end{equation} 
To ensure the existence of weak solutions of (RIVP), via path-wise use of Kru\v{z}kov's theorem,  we impose the following conditions on the initial data and the source term.
\begin{enumerate}[label=(A\arabic*)]
\item There exists a constant $M_1>0$, such that 
$ \| u^0(\cdot,\xi(\omega))\|_{L^{\infty}(0,1)}\leq M_1$, $\p$-a.s. $\omega \in \Omega$ and we have $ u^0\in L^2(\Omega; L^2(0,1))$.
\item There exists a constant $M_2>0$, such that  $\|S(\cdot,\cdot, \xi(\omega))\|_{L^\infty((0,T)\times (0,1))} \leq M_2$, $\p$-a.s. $\omega \in \Omega$ and 
$S \in L^2(\Omega;L^2((0,T)\times (0,1))$, $S(\cdot,\cdot, \xi(\omega))\in C^1( [0,T]\times [0,1]_{per})$, $\p$-a.s. $\omega\in \Omega$ .
\end{enumerate}
To ensure uniqueness of a weak solution of (RIVP) we apply the well-known  entropy criterion. We say that a convex function
$\eta \in C^1(\R)$ and $q\in C^1(\R)$ form an entropy/entropy-flux pair ($\eta,q$), if they satisfy $q'=\eta' f'$.
Following the definition in \cite{Mishra} we define random entropy solutions, which correspond $\P$-a.s. to entropy solutions as in the deterministic case. 
\begin{definition}[Random entropy solution]
We call a weak solution ${u\in L^2(\Omega;L^1((0,T)\times (0,1))}$, with $ u(\cdot,\cdot,\xi(\omega)) \in L^\infty((0,T)\times(0,1)),~ \p\text{-a.s. } \omega\in \Omega,$  
a \textbf{random entropy solution of  (RIVP)}, if the inequality
\begin{align} \label{entropySol}
&\int \limits_0^T \int \limits_0^1 \eta(u(t,x,{\xi(\omega})) ) \partial_t \phi(t,x)  + q(u(t,x, \xi(\omega)) )\partial_x \phi(t,x)  \mathrm{d}x \mathrm{d}t \nonumber   
 \\ + \int \limits_0^T &\int \limits_0^1S (t,x,\xi(\omega))  \eta'(u(t,x,{\xi(\omega}))) \phi(t,x)~ \mathrm{d}x \mathrm{d}t 
 + \int \limits_0^1 \eta (u^0(x, \xi(\omega)) ) \phi(0,x)~\mathrm{d}x  \geq 0
\end{align}
holds $\P$-a.s.~$\omega\in \Omega$, for all $\phi \in C_c^{\infty}([0,T)\times [0,1]_{per}, \R_+)$ and for all entropy/entropy-flux pairs $(\eta,q)$. 
\end{definition}
\begin{remark} ~
 For $u(\cdot,\cdot,\xi(\omega)) \in L^\infty((0,T)\times(0,1)),~ \p\text{-a.s. } \omega\in \Omega$, we have 
$$ \partial_t \eta(u(\cdot,\cdot,\xi(\omega)))+ \partial_x q(u(\cdot,\cdot,\xi(\omega)))-S (\cdot,\cdot,\xi(\omega))  \eta'(u(\cdot,\cdot,{\xi(\omega}))) \leq 0 $$ 
in the distributional sense. Therefore, this distribution has a sign and,
thus, is a measure. Then, we may replace the smooth test functions in Definition \ref{entropySol} by Lipschitz continuous ones, 
cf. \cite[Section 4.5]{dafermos2005hyperbolic}.
\end{remark}
In \cite{Mishra} it is shown, that under the conditions (A1) - (A2), there exists a unique random entropy solution $u \in L^2(\Omega;C((0,T);L^1(0,1)))$ of (RIVP). 
The approach relies on the path-wise existence and uniqueness of entropy solutions of the scalar conservation law and the following path-wise mapping properties 
of the data-to-solution operator $\mathcal{S}(t): u^0(\cdot,\xi(\omega)) \mapsto \mathcal{S}(t) u^0(\cdot, \xi(\omega))$. These can be found in \cite{kruvzkov1970first} 
and will be useful for the proof of Lemma \ref{bochner} below.
\begin{itemize}
\item $L^1(0,1)$-contraction:
\begin{align*}
 \|\mathcal{S}(t)u^0(\cdot,\xi(\omega)) - \mathcal{S}(t)v^0(\cdot,\xi(\omega))\|_{L^1(0,1)} 
\leq  \|u^0(\cdot,\xi(\omega))- v^0(\cdot,\xi(\omega))\|_{L^1(0,1)},
\end{align*}
for almost all $t \in (0,T)$, $\p$-a.s. $\omega \in \Omega$.
\item $L^\infty(0,1)$-boundedness: there exists a constant $M_3=M_3(T, M_1, M_2)>0$, such that  
$$ \|\mathcal{S}(t)u^0(\cdot,\xi(\omega))\|_{L^\infty(0,1)}\leq  M_3, $$
for almost all $t \in (0,T)$, $\p$-a.s. $ \omega \in \Omega$.
\end{itemize}
\begin{remark}
Although we apply Kru\v{z}kov's notion of entropy solutions for the well-posedness of (RIVP), our a posteriori estimates rely on the relative entropy framework of Dafermos instead of Kru\v{z}kov's
doubling of variables technique.
This has three advantages.
Firstly,  for smooth solutions, upper bounds for the numerical error that result from Kru\v{z}kov's theory are only of half-order, i.e. have a rate of convergence of ${(p+1)}/{2}$,
when $p\in \N$ is the polynomial degree, cf. \cite{gosse2000two, kroner2000posteriori}.
In contrast, the worst case for error bounds for smooth solutions obtained using the relative entropy framework is loosing one order, i.e.  the rate of convergence is $p$.
Secondly, the resulting residual of the relative entropy framework allows for an orthogonal decomposition into a spatial (deterministic) residual and a stochastic residual.
This decomposition relies on a discretization error estimate for the stochastic Galerkin system and, thus, cannot be obtained when using Kru\v{z}kov estimates.
Finally, the relative entropy framework requires only one strictly convex  entropy/entropy flux pair and, hence, is extendible to the systems case.
\end{remark}
In the following sections we are dealing with the Bochner space $L^2(\Omega;L^2((0,T)\times (0,1)))$, instead of $L^2(\Omega;L^1((0,T)\times (0,1)))$ as in \cite{Mishra}. 
To ensure that the Bochner integrals are well-defined we prove the following 
\begin{lemma} \label{bochner}
Under the conditions (A1) - (A2) the mapping 
$$(\Omega, \mathcal{F}) \ni \omega \mapsto u(\cdot,\cdot,\xi(\omega)) \in \big(L^2((0,T)\times(0,1)),\mathcal{B}(L^2((0,T)\times(0,1) ) )\big)$$ is measurable. 
\end{lemma}
\begin{proof}
 We use the  interpolation inequality 
 $\|f\|_{L^2}^2\leq \|f\|_{L^1}\|f\|_{L^\infty}$ 
 to estimate for almost all $t \in (0,T)$ and 
 $u_0,v_0 \in L^{\infty}(0,1)$, 
\begin{align*}
 \| \mathcal{S}(t)u_0- \mathcal{S}(t)v_0 \|_{L^2(0,1)} & 
 \leq \| \mathcal{S}(t)u_0- \mathcal{S}(t)v_0 \|_{L^1(0,1)}^{1/2} \| \mathcal{S}(t)u_0- \mathcal{S}(t)v_0 \|_{L^\infty(0,1)}^{1/2}
\\ 
& \leq (2M_3)^{1/2} \| \mathcal{S}(t)u_0- \mathcal{S}(t)v_0 \|_{L^1(0,1)}^{1/2} 
\\ 
& \leq (2M_3)^{1/2} \|u_0- v_0 \|_{L^1(0,1)}^{1/2}
\\ 
& \leq (2M_3)^{1/2} \| u_0- v_0 \|_{L^2(0,1)}^{1/2}.
\end{align*}
Therefore, the data-to-solution operator $\mathcal{S}(t): L^2(0,1) \to L^2(0,1)$ is H\"{o}lder continuous with exponent $\nicefrac{1}{2}$
for almost all $t \in (0,T)$.
We immediately obtain  that the mapping  $\mathcal{S}: L^2(0,1) \ni u_0(\cdot) \mapsto \mathcal{S}(\cdot)u_0(\cdot) \in  L^2((0,T)\times (0,1)), $ 
is also H\"{o}lder continuous with exponent $\nicefrac{1}{2}$.  
Using the fact that the Borel $\sigma$-algebra $\mathcal{B}(L^2(0,1))$  is the smallest $\sigma$-algebra containing all open subsets of $L^2(0,1)$ 
and that $\mathcal{S}^{-1}(B)$ is open  for $B \in \mathcal{B}(L^2((0,T)\times(0,1))$, it follows by the continuity of $\mathcal{S}$, that $\mathcal{S}$ is a measurable mapping 
from  $\big ( L^2(0,1), \mathcal{B}(L^2(0,1)) \big)$ to $\big(L^2((0,T)\times(0,1)), \mathcal{B}(L^2((0,T)\times(0,1)) \big)$.
For $u^0$ as in (A1) the map $(\Omega,\mathcal{F}) \ni \omega \mapsto u^0(\cdot,\xi(\omega)) \in \big( L^2(0,1),\mathcal{B}(L^2(0,1) \big)$ is measurable.
Thus, we have that $(\Omega, \mathcal{F}) \ni \omega \mapsto u(\cdot,\cdot,\xi(\omega))= S(\cdot)u^0(\cdot,\xi(\omega)) \in \big(L^2((0,T)\times(0,1)),\mathcal{B}(L^2((0,T)\times(0,1) ) \big)$ 
 is measurable as composition of measurable functions.
\end{proof}

Let us now introduce the Stochastic Galerkin (SG) method. We first expand the solution of (RIVP) into a generalized Fourier series using  a suitable orthonormal basis.
Let ${\{\Psi_i(\xi(\cdot))\}_{i\in \N_0}: \Omega \to \R}$ be a  $L^2(\Omega)$-orthonormal basis with respect to the density function $p_\xi$, i.e. for $i,j\in \N_0$ we have
\begin{equation}
\begin{aligned}\label{orthogonal}
\Big\langle \Psi_i, \Psi_j \Big\rangle := \E(\Psi_i \Psi_j) &= \int \limits_{\Omega} \Psi_i(\xi(\omega))\Psi_j(\xi(\omega))~ \mathrm{d}\p(\omega) 
\\ &= \int \limits_{\R} \Psi_i(x)\Psi_j(x) p_\xi(x) ~\mathrm{d}x=\delta_{i,j}.
\end{aligned}
\end{equation}
Following \cite{bgrll2017,Wiener}, the random entropy solution $u \in L^2(\Omega;L^1((0,T)\times (0,1))$ of (RIVP) can be written as
\begin{align} \label{series}
u(t,x,\xi(\omega))= \sum \limits_{n=0}^\infty u_n(t,x) \Psi_n(\xi(\omega)).
\end{align}
The deterministic Fourier modes $u_n=u_n(t,x)$ in (\ref{series}) are defined by,
\begin{align*}
u_n(t,x)=\E(u(t,x,\xi(\cdot)) \Psi_n(\xi(\cdot))), \qquad\forall n\in \N_0.
\end{align*}
From the Fourier modes we can immediately extract the expectation and variance of $u$, namely 
$$\E(u(t,x,\xi(\omega)))= u_0(t,x) \text{ and } \Var(u(t,x,\xi(\omega)))= \sum \limits_{n=1}^\infty u_n(t,x)^2.$$ 
For the error estimates in Section 3 we need an additional assumption on the choice of the orthonormal basis.
\begin{assumption}\label{assumptionOrth}
 We assume that the elements of the $L^2(\Omega)$-orthonormal basis $\{\Psi_n( \xi(\cdot))\}_{n=0}^\infty$ (w.r.t. the scalar product (\ref{orthogonal}))  are essentially bounded 
 in $L^\infty(\Omega)$, i.e., for all $n \in \N_0$ we have  $$ \|\Psi_n( \xi(\cdot))\|_{L^\infty(\Omega)} < \infty.$$ 
 \end{assumption}
Let us mention three examples of orthonormal bases satisfying Assumption \ref{assumptionOrth}.
\begin{example} \label{polys} ~
\begin{enumerate}
\item The classical generalized polynomial chaos approach uses global polynomials which satisfy (\ref{orthogonal}) for the corresponding probability density function. 
In this case one chooses Legendre polynomials if random variables are uniformly distributed, or Jacobi polynomials for a beta distribution. See \cite{XiuKarniadakis} for more details. For a normal distribution the corresponding Hermite polynomials do not satisfy Assumption \ref{assumptionOrth}, however one can scale the Hermite polynomials by $e^{-\xi(\omega)/2}$ to obtain the so-called Hermite functions. Due to the Cram\'{e}r inequality this orthonormal system is again bounded, cf. ~\cite[p. 207]{MR698780}.
\item Instead of choosing a polynomial basis we can choose a discontinuous Haar-Wavelet basis as in \cite{LEMAITRE200428}.
\item Another approach is the multi-element approach.   For its illustration, we assume that $\Omega= [0,1]$ and decompose $\Omega$ into $2^{N_e}$, $\N_e \in \N_0$, elements:
$$ \Omega =\bigcup \limits_{l=0}^{2^{N_e}-1} [2^{-N_e}l, 2^{-N_e}(l+1)].$$ 
We introduce on each stochastic element $I_l^{N_e}:= [2^{-N_e}l, 2^{-N_e}(l+1)]$ the shifted and scaled family of orthonormal polynomials:  
$$\Psi_{n,l}^{N_e}(\xi(\omega))= \begin{cases} 2^{N_e/2} \Psi_n(2^{N_e}\xi(\omega)-l), ~~& \xi(\omega) \in I_l^{N_e}, \\ 0, & \text{else} \end{cases} $$
for all $n=0,\ldots,N$, $l=0,\ldots, 2^{N_e}-1$. Here, $\Psi_n$ are either the classical orthonormal polynomials from (1), or a Wavelet basis as in (2). If we introduce the indices $m=(N+1)l+i$, $i=0,\ldots, N$, $\tilde{N}:=(N+1)2^{N_e}$, 
define $\Psi_m := \Psi_{n,l}^{N_e}$, $u_m := u_{n,l}^{N_e}$, we may write (cf. \cite{ MR3269040, MR3671657,MR1952371}): 
$$ u(t,x,\xi(\omega)) = \lim \limits_{N_e, N \to \infty} \sum \limits_{m=0}^{\tilde{N}} u_m(t,x) \Psi_m(\xi(\omega)) \qquad \text{in } L^2(\Omega).$$
\end{enumerate}
\end{example}
We now truncate the infinite series (\ref{series}) at $N\in \mathbb{N}_0$: 
\begin{align} \label{truncSeries}
\sum \limits_{m=0}^N u_m(t,x) \Psi_m(\xi(\omega)),\qquad (t,x,\omega) \in \Omega_{t,x} .
\end{align}
To obtain a weak formulation of (RIVP) on the finite-dimensional space with respect to $\omega\in \Omega$, we insert (\ref{truncSeries})  
and test the resulting equation against $\{\Psi_i(\xi)\}_{i=0}^ N$. Using the orthogonality relation (\ref{orthogonal}) yields the truncated Stochastic Galerkin system
\begin{align}\label{SG} \tag{S}
\partial_t u_l(t,x) + \Big\langle \partial_x f(\sum \limits_{n=0}^N u_n(t,x) \Psi_n(\xi)), \Psi_l \Big\rangle =\Big\langle S, \Psi_l \Big\rangle,~\qquad \forall ~l=0,\ldots, N.
\end{align}
With a slight abuse of notation we define $\fu:=(u_0,\ldots,u_N)^{\top} \in \R^{N+1}$, $\fS\in \R^{N+1}$, $\fS_i:= \int \limits_\Omega S \Psi_i ~ \domega$  and   $\ff: \R^{N+1} \to \R^{N+1},
\ff(\fu)_i:= \int \limits_{\Omega} f(\sum \limits_{n=0}^N u_n \Psi_n) \Psi_i ~\domega  $. We then may write the initial value problem for (\ref{SG}) in the conservation form
\begin{equation}\tag{SG}
\begin{aligned}\label{SGSystem}
&\partial_t \fu+ \partial_x \ff(\fu)=\fS, \qquad &(t,x)&\in (0,T)\times [0,1]_{per}, \\
 &\fu^{0}= \Big(\Big\langle u^0, \Psi_l \Big\rangle\Big)_{l=0}^N, &x&\in [0,1]_{per}. 
\end{aligned}
\end{equation}
We give two examples for the structure of the (\ref{SGSystem}) system, for global orthonormal polynomials.
\begin{example} \label{example1}
\item 
\begin{enumerate}[label=(\alph*)]  
\item We consider the linear advection equation, i.e., $f(u)=au$ for $a\in \R$. In this case the  system  (\ref{SGSystem}) decouples into $N+1$ scalar advection equations. 
This is due to
\begin{align*}
\ff(\fu)_i= a \int \limits_{\Omega} \sum \limits_{n=0}^N u_n \Psi_n \Psi_i~ \domega= a u_i,
\end{align*}
and therefore  (\ref{SGSystem}) takes the form
\begin{align*}
\partial_t \fu+ a \partial_x \fu=\fS.
\end{align*}
\item For  Burgers' equation, i.e. (RIVP) with $f(u)= \frac{u^2}{2}$, we use the symmetric triple product matrices $[C_k]_{i,j=0}^N:= \int \limits_{\Omega} \Psi_j\Psi_i\Psi_k~ \mathrm{d}\p(\omega)$,
which are well-defined due to Assumption \ref{assumptionOrth}.
They enable us to rewrite the (\ref{SGSystem}) system of Burgers' equation as
\begin{align*}
\partial_t \fu_k+\frac{1}{2}\partial_x (\fu^\top C_k \fu)= \fS_k,
\end{align*}
for all $k=0,\ldots,N$. Therefore the system (\ref{SGSystem}) of Burgers' equation is a hyperbolic system but highly coupled, cf. \cite{Despres2013}.
\end{enumerate}
\end{example}
\section{Spatio-temporal-stochastic reconstructions and a posteriori error analysis}
In this section we prove the main theorem of this contribution. We first consider space-time reconstructions for fully discrete Discontinuous Galerkin (DG) schemes applied to the system 
(\ref{SGSystem}). We then introduce space-time-stochastic reconstructions to prove an a posteriori error estimate for the random scalar conservation law (RIVP), 
using the relative entropy framework of Dafermos \cite{dafermos2005hyperbolic}. Finally we prove the orthogonal decomposition of the space-time-stochastic residual into a spatial and a stochastic residual. 
\subsection{The DG scheme for the SG-system and space-time reconstructions}
\label{reconstDG}
We shortly recall the Discontinuous Galerkin spatial discretization as for example in \cite{cockburn1998runge,CockburnShu2001}.
Let $0=x_0<x_1 \ldots <x_M=1$ be a quasi-uniform triangulation of [0,1] and $0=t_0<t_1<\ldots< t_{N_t}=T$ be a  temporal decomposition of $[0,T]$. We identify $x_0=x_M$ to
account for the periodic boundary conditions, set $\Delta t_n :=( t_{n+1}-t_n)$, 
for the temporal mesh and 
$h_k= (x_{k+1}-x_k)$, 
for the spatial mesh. 
We now define the (spatial) piecewise polynomial DG spaces for $p\in \N_0$:
\begin{align*}
V_p^s := \{\fw:[x_0,x_M]\to \R^{N+1}~|~ \fw\mid_{(x_{i-1},x_i)} \in \P_p((x_{i-1},x_i),\R^{N+1}),~1\leq i\leq M \}.
\end{align*}
After spatial discretization of (\ref{SGSystem}) we obtain the following semi-discrete scheme for the discrete solution $\fu_h \in C^1([0,T); V_p^s)$. 
\begin{equation}\tag{DG-SG} 
\begin{aligned} \label{semidiscrete}
\sum \limits_{i=0}^{M-1} \int \limits_{x_i}^{x_{i+1}}  \partial_t  \fu_h \cdot \fpsi_h~\mathrm{d}x  
=& \sum \limits_{i=0}^{M-1} \int \limits_{x_i}^{x_{i+1}} ( \ff(\fu_h) \cdot \partial_x  \fpsi_h+  \fS \cdot \fpsi_h )~\mathrm{d}x     \vspace*{8pt}
-\sum \limits_{i=0}^{M-1} \G(\fu_h(x_i^-),\fu_h(x_i^+) )\cdot  \dbracei \fpsi_h \dbraceo_i,
\\ \fu_h(t=0)=&~\mathcal{I}_{V_p^s} \fu^0,
\end{aligned}
\end{equation}
for all $\fpsi_h \in V_p^s$. Here, $\mathcal{I}_{V_p^s}$ is an interpolation or projection operator, $\\{\G: \R^{N+1}\times \R^{N+1} \to \R^{N+1}}$ denotes a numerical flux, 
the spatial traces are defined as $\fpsi (x^{\pm}):= \lim \limits_{h\searrow 0} \fpsi(x \pm h)$ and $\dbracei \fpsi_h \dbraceo_i:=(\fpsi_h(x_i^-)-\fpsi_h(x_i^+))$ are jumps. 
To account for the periodic boundary condition, we set $\fu_h(x_0^-) = \fu_h(x_M^-), ~ \fu_h(x_M^+) = \fu_h(x_0^+)$.
The initial-value problem (\ref{semidiscrete}) can now be solved numerically by any single- or multi-step method. 
We focus on $\mathcal{K}$-order Runge-Kutta time-step methods as in \cite{CockburnShu2001}. In writing down the method we denote by $L_h(\fu_h(t,\cdot))$ the right-hand side of (\ref{semidiscrete}), with the operator $L_h: V_p^s \to V_p^s$  being defined appropriately. Furthermore, $\Lambda \Pi_h\: \R^{N+1} \to \R^{N+1}$ is the TVBM minmod slope limiter from \cite{CockburnShu2001}. Then, the complete $S$-stage time-marching algorithm for given $n$-th time-iterate $\fu_h^n\in V_p^s $ reads as follows:
\begin{enumerate}
\item Set $\fu_h^{(0)}$ = $\fu_h^n$.
\item For $j=1,\ldots, S$ compute the auxiliary functions:
$$ \fu_h^{(j)} =  \Lambda \Pi_h\Big( \sum \limits_{l=0}^{j-1} \alpha_{jl}\fw_h^{jl}\Big), \quad \fw_h^{jl}=u_h^{(l)} + \frac{\beta_{jl}}{\alpha_{jl}} \Delta t_n L_h(\fu_h^{(l)}).$$
\item Set $\fu_h^{n+1} = \fu_h^{(S).}$
\end{enumerate}
The parameters $\alpha_{jl}, \beta_{jl}$ satisfy the conditions $\alpha_{jl}\geq 0$, $\sum \limits_{l=0}^{j-1} \alpha_{jl}=1$ , and if $\beta_{jl} \neq 0$, then $\alpha_{jl} \neq 0$ for all $j=1,\ldots, \mathcal{K}$, $l=0,\ldots,j$.

The relative entropy framework requires one quantity which is at least Lipschitz continuous in space and time. We cannot expect that the entropy solution of (RIVP) satisfies this condition,
therefore, we reconstruct the numerical solution of (\ref{semidiscrete}) to obtain 
a Lipschitz continuous function in space and time.
Following \cite{DG_16}, we define the \textbf{temporal reconstruction} $\hat{\fu}^t$, as a $C^0$- or even $C^1$-function (depending on the polynomial degree) 
which is piecewise polynomial.
From a practical perspective it makes sense to choose the polynomial degree of the reconstruction such that it matches the formal convergence order of the time-stepping method.
This choice ensures that the temporal residual has the same order of convergence as the temporal discretization error \cite[Theorem 13]{DG_16}.
Thus, the error estimator enables us to detect whether the convergence order is reduced, e.g. in case $f$ is not sufficiently smooth for the formal convergence order to be reached.
Let $\{ \fu_h^0, \ldots, \fu_h^{N_t} \}$ 
be a sequence of approximate solutions of (\ref{semidiscrete}) at points $\{t_n\}_{n=0}^{N_t}$ in time. 
For the reconstruction in time we define for any vector space $V$ the spaces of piecewise polynomials in time of degree $r$ by
\begin{align*}
V_r^t((0,T);V):=\{\fw:[0,T]\to V~|~\fw\mid_{(t_n,t_{n+1})} \in \P_r((t_n,t_{n+1}),V) \}.
\end{align*}
Using the methodology proposed in \cite{DG_16}, which consists of Hermite interpolations on each time interval $[t_n,t_{n+1}]$, 
we construct a temporal reconstruction \\$\hat{\fu}^t \in V_r^t((0,T);V_p^s)$.
With the temporal reconstruction at hand, we now define the space-time reconstruction of the DG-solutions of (\ref{SGSystem}). 
The analysis in \cite{DG_16} requires numerical fluxes $\G$ which admit a special representation.
In particular, there needs to exist a  locally Lipschitz function $\fw: \R^{N+1} \times \R^{N+1} \to \R^{N+1}$, with the additional property $\fw(\fu,\fu)=\fu$,
such that  $\G$ can be expressed as
\begin{align}\label{assum:flux}
\G(\fu,\fv)=\ff(\fw(\fu,\fv)), \quad \forall \fu,\fv \in  \R^{N+1}.
\end{align}
For our numerical computations we consider the upwind flux with $\fw(\fu,\fv)= \fu$ and the
Lax-Wendroff flux with $\fw(\fu,\fv)= \frac{\fu +\fv}{2} - \frac{\Delta t}{2 h}(\ff(\fu)-\ff(\fv))$, both satisfying (\ref{assum:flux}).

We now define the spatial reconstruction which is applied to the temporal reconstruction $\hat{\fu}^t(t,\cdot)$ for each $t\in (0,T)$ using the function $\fw$ (cf. \cite{DG_16, GMP_15}). 
\begin{definition}[Space-time reconstruction]\label{stReconstruct}
Let $\hat{\fu}^t$ be the temporal reconstruction of a sequence $\{\fu_h^n\}_{n=0}^{N_t}$ of solutions of the fully discrete scheme of (\ref{semidiscrete}) using a numerical flux satisfying (\ref{assum:flux}). The \textbf{space-time reconstruction} $\hat{\fu}^{st}(t,\cdot) \in V_{p+1}^s$ is defined as the solution of
\begin{align*}
\sum \limits_{i=0}^{M-1} \int \limits_{x_i}^{x_{i+1}}(\hat{\fu}^{st}(t,\cdot)- \hat{\fu}^{t}(t,\cdot)) \cdot \fpsi ~\mathrm{d}x &=0 \quad \forall \fpsi \in V_{p-1}^s,
\\ 
\hat{\fu}^{st}(t,x_k^{\pm})&= \fw(\hat{\fu}^{t}(t,x_k^-),\hat{\fu}^{t}(t,x_k^+)) \quad \forall~ k=0,\ldots,M.
\end{align*}
\end{definition}
\noindent We have the following property of the space-time reconstruction.
\begin{lemma} \label{dx}
Let $\hat{\fu}^{st}$ be the space-time reconstruction from Definition \ref{stReconstruct}. For each $t\in (0,T)$, the function $\hat{\fu}^{st}(t,\cdot)$ is well defined. Moreover,
\begin{align*}
\reconst \in W_\infty^1((0,T);V_{p+1}^s\cap C^0[0,1]_{per}).
\end{align*}
\end{lemma}
\begin{proof}
Cf. \cite[Lemma 24]{DG_16}.
\end{proof} 

Since $\hat{\fu}^{st}$ is Lipschitz continuous in space and time, we can compute the space-time residual.
\begin{definition}[Space-time residual]
We define $\fR^{st}:= \partial_t \hat{\fu}^{st}+ \partial_x \ff(\hat{\fu}^{st})- \fS\in L^2((0,T)\times(0,1);\R^{N+1})$ to be the \textbf{space-time residual}.
\end{definition}
\subsection{Space-time-stochastic reconstructions and a posteriori estimate for the random scalar conservation law  }
Expanding the space-time reconstruction $\reconst$ from Section \ref{reconstDG} in the finite orthonormal system
 \\ $\{\Psi_i(\xi(\omega))\}_{i=0}^N$ 
enables us to consider the so-called space-time-stochastic residual, which is a crucial part for the upcoming a posteriori error estimate. Before defining the space-time-stochastic residual, we first give a definition of what is our approximation of a random entropy solution of (RIVP). For simplicity, 
we will write $u(t,x,\omega)$ for $u(t,x,\xi(\omega))$.
\begin{definition}[Numerical solution of  (RIVP)] \label{numsol}
Let $\{ \fu_h^0, \ldots, \fu_h^{N_t} \}$ be the sequence of solutions of (\ref{semidiscrete}) at points $\{t_n\}_{n=0}^{N_t}$ in time. 
For  $n=0,\ldots,N_t$, $u_h^n(x,\omega)= \sum \limits_{i=0}^N (\fu_h^n(x))_i \Psi_i(\omega)$ is called the \textbf{numerical solution of (RIVP)}.
\end{definition} 
\begin{definition}[Space-time-stochastic residual] \label{stsreconst} ~\\
Let $\hat{u}^{sts}(t,x,\omega):= \sum \limits_{l=0}^N (\reconst)_l (t,x) \Psi_l(\omega)\in L^2(\Omega) \otimes (W_\infty^1(0,T); (\overline{V}_{p+1}^s\cap C^0([0,1]_{per}))$ 
be the \textbf{space-time-stochastic reconstruction}, where  
\begin{align*}
\overline{V}_{p}^s:= \{w:[x_0,x_M]\to \R~|~ w\mid_{(x_{i-1},x_i)} \in \P_{p}((x_{i-1},x_i),\R),~ 1\leq i\leq M\}. 
\end{align*}
We define the \textbf{space-time-stochastic residual} by
\begin{align}
\label{stsResidual}
\cR^{sts} :=  \partial_t \hat{u}^{sts}+ \partial_x f(\hat{u}^{sts})- S \in L^2((0,T)\times (0,1) \times \Omega;\R).
\end{align}
\end{definition}
For the proof of the main theorem of this article, we need the notion of the relative entropy and the relative entropy flux. We follow the definition in \cite[Section 5.2]{dafermos2005hyperbolic}.
\begin{definition}[Relative entropy and entropy flux]
\label{relEntropy}
We define the \textbf{relative entropy $\eta(\cdot|\cdot) :\R \times \R  \to \R$} and the  \textbf{relative entropy flux} $q(\cdot|\cdot) :\R \times \R  \to \R$ by:
\begin{align}
&{\eta}(u|v)= {\eta}(u)-{\eta}(v)- \eta'(v)(u-v), \\
& {q}(u|v)= {q}(u)-{q}(v)- \eta'(v)({f}(u)-{f}(v)).
\end{align}
\end{definition}

\begin{example}[relative entropy ]
\label{exp:relEntropy}
For $\eta(u)=\frac{u^2}{2}$ we compute 
\begin{align}
&{\eta}(u|v)= \frac{u^2-v^2}{2}- v(u-v)= \frac{(u-v)^2}{2}.
\end{align}
This choice is used 
in the proof of Theorem \ref{stochApost} below.
For arbitrary strictly convex $\eta \in C^2(\R)$, an analogous proof leads to very similar estimates which contain constants related to the modulus of convexity of $\eta$, see \cite{DG_16}  for further details. 
\end{example}
Using the relative entropy framework we are able to derive the following a posteriori error bound for (RIVP).
\begin{theorem}[A posteriori error bound for the reconstruction of the numerical solution] \label{stochApost}
Let $u$ be a random entropy solution of (RIVP). Then, the difference between $u$ and the reconstruction $\hat{u}^{sts}$ from Definition \ref{stsreconst} satisfies
\begin{align*}
 \|u(s,\cdot,\cdot)-\hat{u}^{sts}(s,\cdot,\cdot)\|_{L^2((0,1) \times \tilde{\Omega})}^2
\leq &
 \Big(\|\cR^{sts}\|_{L^2((0,s)\times (0,1) \times \tilde{\Omega})}^2+ \|u^0-\hat{u}^{sts}(0,\cdot,\cdot)\|_{L^2((0,1)\times \tilde{\Omega})}^2  \Big) 
 \\& \times  \exp\Big( \int \limits_0^s \Big(C_{f''} \|\partial_x \hat{u}^{sts}(t,\cdot,\cdot)\|_{L^\infty((0,1) \times \tilde{\Omega})}+ \frac{1}{4}\Big) ~\mathrm{d}t \Big),
\end{align*}
for $0\leq s\leq T$ and for any $\P$-measurable set $\tilde{\Omega} \subset \Omega$. Here, $C_{f''}:= \max \limits_{ u \in {M}} \frac{|f''(u)|}{2}$,
with ${M}:=\emph{Conv}([-M_3,M_3]\cup \emph{Ran}(\hat{u}^{sts}))$, where \emph{Conv} denotes the convex hull and \emph{Ran} the image of $\hat{u}^{sts}$.
\end{theorem}
\begin{remark} \label{blowup} ~
\begin{enumerate} 
\item All quantities in the upper bound of Theorem \ref{stochApost} are computable during a numerical simulation.
However, the computation of \emph{Ran}$(\hat{u}^{sts})$ can be very expensive, in particular for large DG polynomial degrees. 
 \item  Due to Lemma \ref{dx} every component of $\reconst$ is Lipschitz continuous in space and therefore has a bounded derivative.
 Combined with Assumption \ref{assumptionOrth} we obtain that every summand of $$\partial_x \hat{u}^{sts}= \sum \limits_{l=0}^N \partial_x (\reconst)_l (t,x) \Psi_l(\omega)$$
 is in $L^\infty((0,1) \times \Omega)$, which ensures that
$$ \|\partial_x \hat{u}^{sts}(s,\cdot,\cdot)\|_{L^\infty((0,1) \times \Omega)} < \infty ~\forall ~s\in (0,T].$$\\
Note also that we actually only need  $\|\partial_x \hat{u}^{sts}(s,\cdot,\cdot)\|_{L^\infty((0,1) \times \tilde \Omega)}<\infty$.  Thus, we could remove Assumption \ref{assumptionOrth} 
and formulate Theorem \ref{stochApost} only for those subsets $\tilde \Omega \subset \Omega$ for which \\ $ \|\Psi_n( \xi(\cdot))\|_{L^\infty(\tilde \Omega)} < \infty$ for all $n \in \N_0$.

\item As described in \cite{pettersson2015polynomial}, the solution of system (\ref{SGSystem}) may exhibit discontinuities in the spatial variables, 
even if the solution of the original problem is smooth.
In case the exact solution of the (\ref{SGSystem})-system is discontinuous, 
the quantity $\|\partial_x \hat{u}^{sts}\|_{L^\infty(\Omega\times [0,1])}$ is expected to scale like $h^{-1}$ 
and hence the estimator will blow up for $h \to 0$ in the vicinity of discontinuities.
\item  For linear equations $f(u)=au$, $a\in \R$, we have that $C_{f''}=0$. Therefore, for linear equations we expect no blow-up of the estimator for $ h \to 0$ even in case the solution is discontinuous.
\end{enumerate}
\end{remark}
\begin{remark}[Localization of the error estimators]
Due to the point-wise definition of the random entropy solution, the choice of $\tilde{\Omega}$ is arbitrary. 
This enables us to localize the error estimator in the stochastic variable. With some more effort it is also possible to localize the error estimator in the physical space 
in a wave-cone around the region of interest, cf. the proof of Theorem 5.2.1 in \cite{dafermos2005hyperbolic}. For scalar problems and a $L^1$-theory a restriction to a smaller wave cone is possible, see \cite{gosse2000two}. The $L^2$-theory as in our case is more challenging for the scalar case and, as already noted in Remark \ref{blowup} (c), the estimator will blow up in the vicinity of a discontinuity. However, in contrast to the Kru\v{z}kov framework, the relative entropy framework is extendable to systems of conservation laws and allows us to derive similar error estimators  for the case of systems.
\end{remark}
\begin{proof}[Proof of Theorem \ref{stochApost}]

Because $u$ is a pointwise entropy solution of (RIVP) $\p$-a.s. $\omega \in \Omega$, we can integrate the entropy inequality (\ref{entropySol}) 
with respect to any $\P$-measurable $\tilde{\Omega} \subset \Omega$. Thus, we have for any  nonnegative Lipschitz continuous test function $\phi=\phi(t,x)$ the inequality
\begin{align} \label{StochentropyIneq}
0 \leq \int \limits_{\tilde{\Omega}}  &\int \limits_0^T \int \limits_0^1 \eta(u)  \partial_t \phi  + q(u)\partial_x \phi \mathrm{d}x \mathrm{d}t  \mathrm{d}\P(\omega)+ \int \limits_{\tilde{\Omega}} \int \limits_0^T \int \limits_0^1 S  \eta'(u) \phi~ \mathrm{d}x \mathrm{d}t \mathrm{d}\P(\omega) 
\\
 + &  \int \limits_{\tilde{\Omega}}\int \limits_0^1 \eta (u^0 ) \phi(0,x)~\mathrm{d}x \mathrm{d}\P(\omega). \notag
\end{align} 
Next, we multiply (\ref{stsResidual}) by  $\eta'(\hat{u}^{sts})$. Upon using the chain rule for Lipschitz continuous functions and the relationship $q'=\eta'f'$ we derive the following relation
\begin{align} \label{strongEntropy}
\eta'(\hat{u}^{sts}) \cR^{sts} = \partial_t (\eta(\hat{u}^{sts}))+ \partial_x q(\hat{u}^{sts})- \eta'(\hat{u}^{sts})S.
\end{align}
For the rest of the proof we choose a quadratic entropy function, i.e. we set $\eta(u)= \frac{u^2}{2}$.
We then consider the weak form of (\ref{strongEntropy}), integrate with respect to $\tilde{\Omega} \subset \Omega$ and subtract this from (\ref{StochentropyIneq}) to obtain 
\begin{align*}
0 \leq & \int \limits_{\tilde{\Omega}}\int \limits_0^T \int \limits_0^1 (\eta(u)-\eta(\hat{u}^{sts}) ) \partial_t \phi
+ (q(u)-q(\hat{u}^{sts}) )\partial_x \phi~\mathrm{d}x \mathrm{d}t \mathrm{d}\P(\omega) \\ 
& -\int \limits_{\tilde{\Omega}}\int \limits_0^T \int \limits_0^1 \cR^{sts} \hat{u}^{sts}\phi~ \mathrm{d}x \mathrm{d}t \mathrm{d}\P(\omega)
+\int \limits_{\tilde{\Omega}}\int \limits_0^T \int \limits_0^1 S (u - \hat{u}^{sts})\phi ~\mathrm{d}x \mathrm{d}t \mathrm{d}\P(\omega)
\\ & +\int \limits_{\tilde{\Omega}}\int \limits_0^1 (\eta(u^0)- \eta(\hat{u}_{0}^{sts}) ) \phi(0,x)~\mathrm{d}x \mathrm{d}\P(\omega),
\end{align*}
with $\hat{u}_{0}^{sts}:= \hat{u}^{sts}(0,\cdot,\cdot)$.
Using Definition \ref{relEntropy} of the relative entropy and the relative entropy flux, we may write
\begin{align}\label{someineq}
0 \leq &\int \limits_{\tilde{\Omega}}\int \limits_0^T \int \limits_0^1 (\eta(u|\hat{u}^{sts})+\hat{u}^{sts}(u-\hat{u}^{sts}) ) \partial_t \phi
~ \mathrm{d}x \mathrm{d}t \mathrm{d}\P(\omega)
\nonumber \\ 
& +
\int \limits_{\tilde{\Omega}}\int \limits_0^T \int \limits_0^1 (q(u|\hat{u}^{sts})+ \hat{u}^{sts}(f(u)-f(\hat{u}^{sts})))\partial_x \phi ~\mathrm{d}x \mathrm{d}t \mathrm{d}\P(\omega) \nonumber
\\ & -\int \limits_{\tilde{\Omega}}\int \limits_0^T \int \limits_0^1 \cR^{sts} \hat{u}^{sts}\phi ~\mathrm{d}x \mathrm{d}t \mathrm{d}\P(\omega)
+\int \limits_{\tilde{\Omega}}\int \limits_0^T \int \limits_0^1 S (u - \hat{u}^{sts})\phi~ \mathrm{d}x \mathrm{d}t \mathrm{d}\P(\omega)
\nonumber 
\\ &+ \int \limits_{\tilde{\Omega}}\int \limits_0^1 (\eta(u^0)- \eta(\hat{u}_{0}^{sts}) ) \phi(0,x)~\mathrm{d}x \mathrm{d}\P(\omega).
\end{align}
Using the Lipschitz continuous (in space and time) test function $\phi \hat{u}^{sts} $ in the weak form of \eqref{stochCL} and in (\ref{stsResidual}) we obtain
\begin{align} \label{someeq}
0 =  &\int \limits_{\tilde{\Omega}}\int \limits_0^T \int \limits_0^1 (u- \hat{u}^{sts}) \partial_t (\hat{u}^{sts}\phi ) + (f(u)- f(\hat{u}^{sts})) \partial_x(\hat{u}^{sts}\phi  )
~ \mathrm{d}x \mathrm{d}t \mathrm{d}\P(\omega) \nonumber 
\\ & - \int \limits_{\tilde{\Omega}}\int \limits_0^T \int \limits_0^1  \cR^{sts} \hat{u}^{sts}\phi  ~ \mathrm{d}x \mathrm{d}t \mathrm{d}\P(\omega)
 +\int \limits_{\tilde{\Omega}}\int \limits_0^1 (u^0-\hat{u}_{0}^{sts}) \phi(0,x) \hat{u}_{0}^{sts}~\mathrm{d}x\mathrm{d}\P(\omega).
\end{align}
Using the product rule
\begin{align*}
 &\partial_t (\hat{u}^{sts}\phi ) =  \partial_t \hat{u}^{sts}\phi  + \partial_t \phi  \hat{u}^{sts},
 \\ &\partial_x (\hat{u}^{sts}\phi ) = \partial_x \hat{u}^{sts} \phi  + \partial_x \phi  \hat{u}^{sts}.
\end{align*}
 Combining (\ref{someeq}) with (\ref{someineq}), we obtain
\begin{align*}
0 \leq &\int \limits_{\tilde{\Omega}}\int \limits_0^T \int \limits_0^1 \eta(u|\hat{u}^{sts}) \partial_t \phi  
+ q(u|\hat{u}^{sts})\partial_x \phi  - \partial_t \hat{u}^{sts}(u-\hat{u}^{sts})\phi ~ \mathrm{d}x \mathrm{d}t \mathrm{d}\P(\omega)
\\ &- \int \limits_{\tilde{\Omega}}\int \limits_0^T \int \limits_0^1(\partial_x\hat{u}^{sts} (f(u)-f(\hat{u}^{sts})) \phi 
-S (u - \hat{u}^{sts})\phi  ) ~ \mathrm{d}x \mathrm{d}t \mathrm{d}\P(\omega)
\\ &+ \int \limits_{\tilde{\Omega}}\int \limits_0^1 \eta(u^0|\hat{u}_{0}^{sts}) \phi(0,x) ~ \mathrm{d}x \mathrm{d}\P(\omega).
\end{align*}
Using the fact that 
\begin{align*}
\partial_t \hat{u}^{sts} = -f'(\hat{u}^{sts})\partial_x \hat{u}^{sts} + \cR^{sts} +S,
\end{align*}
we conclude that
\begin{align*}
0 \leq &\int \limits_{\tilde{\Omega}}\int \limits_0^T \int \limits_0^1 \eta(u|\hat{u}^{sts}) \partial_t \phi  + q(u|\hat{u}^{sts})\partial_x \phi  
~ \mathrm{d}x \mathrm{d}t \mathrm{d}\P(\omega)
\\ & -\int \limits_{\tilde{\Omega}}\int \limits_0^T \int \limits_0^1
 [-f'(\hat{u}^{sts})\partial_x \hat{u}^{sts} + \cR^{sts} +S](u-\hat{u}^{sts})\phi  ~  \mathrm{d}x \mathrm{d}t \mathrm{d}\P(\omega)
\\ &- \int \limits_{\tilde{\Omega}}\int \limits_0^T \int \limits_0^1[ \partial_x\hat{u}^{sts} (f(u)-f(\hat{u}^{sts})) \phi - S (u - \hat{u}^{sts})\phi ] 
~ \mathrm{d}x \mathrm{d}t \mathrm{d}\P(\omega)
\\ &+ \int \limits_{\tilde{\Omega}}\int \limits_0^1 \eta(u^0|\hat{u}_{0}^{sts}) \phi(0,x)~\mathrm{d}x \mathrm{d}\P(\omega).
\end{align*}
The last inequality is reformulated as
\begin{align} \label{inequality}
0 \leq &\int \limits_{\tilde{\Omega}}\int \limits_0^T \int \limits_0^1 \eta(u|\hat{u}^{sts}) \partial_t \phi  + q(u|\hat{u}^{sts})\partial_x \phi  
~ \mathrm{d}x \mathrm{d}t \mathrm{d}\P(\omega)
\nonumber
\\& - \int \limits_{\tilde{\Omega}}\int \limits_0^T \int \limits_0^1 \partial_x\hat{u}^{sts} (f(u)-f(\hat{u}^{sts})
-f'(\hat{u}^{sts})(u-\hat{u}^{sts}) ) \phi  ~\mathrm{d}x \mathrm{d}t \mathrm{d}\P(\omega)
\nonumber
\\ & - \int \limits_{\tilde{\Omega}}\int \limits_0^T \int \limits_0^1 \cR^{sts} (u-\hat{u}^{sts}) \phi 
~ \mathrm{d}x \mathrm{d}t \mathrm{d}\P(\omega)
 + \int \limits_{\tilde{\Omega}}\int \limits_0^1 \eta(u^0|\hat{u}_{0}^{sts}) \phi(0,x)~\mathrm{d}x \mathrm{d}\P(\omega).
\end{align}
Up to now the choice of $\phi $ was arbitrary. Now we fix $s>0$ and 
$\epsilon>0$ and define ${\phi}$ as follows
\begin{align*}
{\phi}(\sigma,x):= \begin{cases}
1 \qquad &: \sigma<s,\\
1- \frac{\sigma-s}{\epsilon} &: s<\sigma <s+\epsilon, \\
0 &: \sigma> s+\epsilon.
\end{cases}
\end{align*}
With this particular choice we obtain
\begin{align*}
0 \leq &-\frac{1}{\epsilon} \int \limits_{\tilde{\Omega}}\int \limits_s^{s+\epsilon} \int \limits_0^1 \eta(u|\hat{u}^{sts}) ~\mathrm{d}x \mathrm{d}t \mathrm{d}\P(\omega)
- \int \limits_{\tilde{\Omega}}\int \limits_0^T \int \limits_0^1 \partial_x\hat{u}^{sts} (f(u)-f(\hat{u}^{sts})-f'(\hat{u}^{sts})(u-\hat{u}^{sts}) ) \phi 
~\mathrm{d}x \mathrm{d}t \mathrm{d}\P(\omega)
\nonumber\\ & - \int \limits_{\tilde{\Omega}}\int \limits_0^T \int \limits_0^1 \cR^{sts} (u-\hat{u}^{sts}) \phi  
~\mathrm{d}x \mathrm{d}t \mathrm{d}\P(\omega)
+ \int \limits_{\tilde{\Omega}}\int \limits_0^1 \eta(u^0|\hat{u}_{0}^{sts})~\mathrm{d}x \mathrm{d}\P(\omega).
\end{align*}
Sending $\epsilon \to 0$ we find for all Lebesgue-points $s$ of $\eta(u(\sigma,\cdot, \cdot))$ in $(0,T)$ that
\begin{align*}
0 \leq &-\int \limits_{\tilde{\Omega}}\int \limits_0^1 \eta(u(s,\cdot,\cdot)|\hat{u}^{sts}(s,\cdot,\cdot)) ~\mathrm{d}x  \mathrm{d}\P(\omega)
\\ &- \int \limits_{\tilde{\Omega}}\int \limits_0^s \int \limits_0^1 \partial_x\hat{u}^{sts} (f(u)-f(\hat{u}^{sts})-f'(\hat{u}^{sts})(u-\hat{u}^{sts}) ) 
~\mathrm{d}x \mathrm{d}t \mathrm{d}\P(\omega)
\nonumber
\\ & - \int \limits_{\tilde{\Omega}}\int \limits_0^s\int \limits_0^1 \cR^{sts}(u-\hat{u}^{sts}) ~\mathrm{d}x \mathrm{d}t \mathrm{d}\P(\omega)
+ \int \limits_{\tilde{\Omega}}\int \limits_0^1 \eta(u^0|\hat{u}_{0}^{sts})~\mathrm{d}x \mathrm{d}\P(\omega).
\end{align*}
We then estimate 
$$\int \limits_{\tilde{\Omega}}\int \limits_0^s\int \limits_0^1 \cR^{sts}(u-\hat{u}^{sts}) ~\mathrm{d}x \mathrm{d}t \mathrm{d}\P(\omega)$$
 by Young's inequality. The integral  
$$\int \limits_{\tilde{\Omega}}\int \limits_0^s \int \limits_0^1 \partial_x\hat{u}^{sts} (f(u)-f(\hat{u}^{sts})-f'(\hat{u}^{sts})(u-\hat{u}^{sts}) ) 
~\mathrm{d}x \mathrm{d}t \mathrm{d}\P(\omega)$$ is estimated by 
Taylor's theorem which yields the constant $C_{f''}$. The remaining terms are estimated using the explicit form of $\eta(\cdot,\cdot)$, cf. Example \ref{exp:relEntropy}. Altogether we obtain
\begin{align*}
 & \int \limits_{\tilde{\Omega}}\int \limits_0^1  |u(s,\cdot,\cdot)-\hat{u}^{sts}(s,\cdot,\cdot)|^2 ~\mathrm{d}x  \mathrm{d}\P(\omega) 
 \\ \leq &~C_{f''} \int \limits_0^s \Big( \|\partial_x \hat{u}^{sts}(t,\cdot,\cdot)\|_{L^\infty( (0,1) \times \tilde{\Omega})}  
 \int \limits_{\tilde{\Omega}}\int \limits_0^1 |u-\hat{u}^{sts}|^2 ~\mathrm{d}x  \mathrm{d}\P(\omega) \Big)~ \mathrm{d}t
\\ & +\|\cR^{sts}\|_{L^2((0,s)\times (0,1)\times  \tilde{\Omega} )}^2 
+\frac{1}{4} \int \limits_{\tilde{\Omega}} \int \limits_0^s \int \limits_0^1 |u-\hat{u}^{sts}|^2 ~\mathrm{d}x \mathrm{d}t \mathrm{d}\P(\omega)
\\ & + \|u^0 - \hat{u}^{sts}(0,\cdot,\cdot)\|_{L^2( (0,1)\times \tilde{\Omega} ) }^2.
\end{align*}
Gronwall's inequality yields the assertion.
\end{proof}

We now want to formulate Theorem \ref{stochApost} in terms of the numerical solution $u_h^n$ of (RIVP).  By means of the triangle inequality we obtain directly
\begin{corollary}[A posteriori error bound for the numerical solution] \label{stochApost2}
Let $u$ be a random entropy solution of (RIVP). Then, the difference between $u$ and the numerical solution $u_h^n$ from Definition \ref{numsol} satisfies:
\begin{flalign*}
\|u(t_n,\cdot,\cdot)-{u_h^n}(\cdot,\cdot)\|_{L^2((0,1) \times \tilde{\Omega})}^2
\leq  &~2  \|\hat{u}^{sts}(t_n,\cdot,\cdot)- u_h^n(\cdot,\cdot)\|_{L^2( (0,1) \times \tilde{\Omega})}^2 
\\ &+ 2 \Big(\|\cR^{sts}\|_{L^2( (0,t_n)\times (0,1) \times \tilde{\Omega} )}^2+ \|u^0-\hat{u}^{sts}(0,\cdot,\cdot)\|_{L^2((0,1) \times \tilde{\Omega})}^2 \Big) 
\\ &\times  \exp\Big( \int \limits_0^{t_n} \Big(C_{f''} \|\partial_x \hat{u}^{sts}(t,\cdot,\cdot)\|_{L^\infty((0,1) \times \tilde{\Omega})}+ \frac{1}{4}\Big) ~\mathrm{d}t \Big)
\end{flalign*}
for all $n=0,\ldots N_t$ and for all measurable sets $\tilde{\Omega}\subset \Omega$.
\end{corollary}
\subsection{An orthogonal decomposition of the space-time-stochastic residual}
Due to Corollary \ref{stochApost2} we have a computable upper bound for the space-stochastic error. However, we still cannot distinguish between the errors arising from spatial and stochastic discretization. We therefore want to show a remarkable property of the space-time-stochastic residual $\cR^{sts}$ from Definition \ref{stsreconst}, which indeed allows us to distinguish between spatial and stochastic error.
Namely, the coefficients of the projection of $\cR^{sts}$ onto $\text{span}\{\Psi_0,\ldots,\Psi_N\}$ coincide with the coefficients of $\fR^{st}$.
This can be seen from the following computation:
\begin{align*}
\Big\langle \cR^{sts}, \Psi_l \Big\rangle &= \Big\langle  \partial_t \hat{u}^{sts}+ \partial_x f(\hat{u}^{sts})-S, \Psi_l  \Big\rangle
\\ &= \int \limits_\Omega  \Big(\partial_t \sum \limits_{k=0}^N (\reconst)_k \Psi_k+ \partial_x f(\sum \limits_{k=0}^N (\reconst)_k \Psi_k) -S \Big) \Psi_l ~\mathrm{d}\p (\omega)
\\ & = \partial_t (\reconst)_l + (\partial_x \ff(\reconst))_l - \fS_l = (\fR^{st} )_l.
\end{align*}
On the other hand, for $l>N$,
\begin{align*}
\Big\langle \cR^{sts}, \Psi_l \Big\rangle =  \int \limits_\Omega \big( \partial_x f(\sum \limits_{k=0}^N (\reconst)_k \Psi_k)- S \big) \Psi_l ~\mathrm{d}\p(\omega) =: (\fR^{stoch})_l.
\end{align*}
These properties of $\cR^{sts}$ translate into the following
\begin{theorem}[Orthogonal decomposition of the space-time-stochastic residual]
\label{normOrtho} The space-time-stoch-astic residual $\cR^{sts}$, from Definition \ref{stsreconst},  admits the following orthogonal decomposition in $L^2(\Omega)$,
\begin{align} \label{orthodecomp}
\cR^{sts} = \sum \limits_{i=0}^N (\fR^{st})_i \Psi_i + \sum \limits_{i=N+1}^\infty (\fR^{stoch})_i \Psi_i. 
\end{align}
Further,
\begin{align} \label{orthonorm}
\| \cR^{sts} \|_{L^2((0,s) \times (0,1) \times \Omega)}^2   = & \sum \limits_{i=0}^N \| (\fR^{st})_i \|_{L^2( (0,s) \times (0,1))}^2 +
\sum \limits_{i=N+1}^\infty \| (\fR^{stoch})_i \|_{L^2( (0,s) \times (0,1))}^2 \notag 
\\ =: & \cR^{st} + \cR^{stoch}.  
\end{align}
\end{theorem}
\begin{proof}
Formula (\ref{orthodecomp}) follows from the previous computations and formula (\ref{orthonorm}) is an application of the Pythagorean theorem for the Hilbert space $L^2(\Omega)$.
\end{proof}

\begin{remark}\label{rem:ini}
In the same manner we can find an orthogonal decomposition of the approximation error in the initial condition.
Let us define the orthogonal projection $\Pi_N (u^0):= \sum \limits_{m=0}^N \langle u^0, \Psi_m \rangle \Psi_m$. Then, the Pythagorean theorem implies
\begin{flalign*}
\| u^0-\hat{u}^{sts}(0,\cdot,\cdot)\|_{L^2((0,1) \times {\Omega})}^2 =& \|u^0-\Pi_N (u^0)\|_{L^2((0,1) \times {\Omega})}^2 + \|\Pi_N (u^0) - \hat{u}^{sts}(0,\cdot,\cdot)\|_{L^2((0,1) \times {\Omega})}^2 &
\\ =& \| \sum_{m=N+1}^\infty \langle u^0, \Psi_m\rangle \|_{L^2(0,1)}^2 +  \sum \limits_{m=0}^N \| \langle u^0, \Psi_m \rangle -  \hat{\fu}_m^{st}(0,\cdot)\|_{L^2(0,1)}^2 &
\\  =&: \cE_0^{stoch} + \cE_0^{st} .&
\end{flalign*}
\end{remark}
Theorem \ref{normOrtho} allows us to decompose the error estimator for the space-time-stochastic error into parts quantifying the stochastic and the space-time discretization error, 
respectively.
The stochastic error, introduced by truncating the Fourier series in (\ref{series}), can be quantified by 
$\cE_0^{stoch}  + \cR^{stoch}$. 
The space-time discretization error, i.e., the error which arises by discretising the SG system (\ref{SGSystem}) in space and time, 
can be quantified by  $ \cE_0^{st}  +\cR^{st}$.
The optimality of the spatial residual in the case of linear hyperbolic equations was proven in \cite{DG_16}.
Optimality means that the space-time residual has the same order of convergence as the DG scheme.
However, in the nonlinear case, the optimality of the residual in the first time step is still open. We address this issue in Section 4.
For the stochastic residual we expect, at least for smooth solutions, spectral convergence.
Let us mention how to compute $\fR^{stoch}$ in the case of the linear advection equation and Burgers' equation.
\begin{example} \label{example2}~
\begin{enumerate}[label=(\alph*)]  
\item In the case of the linear advection equation the stochastic residual $(\cR^{stoch})_l$, $l>N$, in (\ref{orthodecomp}) vanishes due to orthogonality. This means that the stochastic error is only inferred from projecting the initial condition onto the orthonormal system, cf. Example \ref{example1} (a). Due to linearity of the advection equation the initial stochastic error does not increase when advancing in time. 
\item In the case of Burgers' equation and classical global orthonormal chaos polynomials, we have
\begin{align*}
\int \limits_\Omega  \partial_x f\Big(\sum \limits_{k=0}^N (\reconst)_k \Psi_k\Big) \Psi_l~ \mathrm{d}\p(\omega) 
= \frac{1}{2} \int \limits_\Omega  \partial_x \Big(\sum \limits_{k=0}^N (\reconst)_k \Psi_k\Big)^2 \Psi_l ~ \mathrm{d}\p(\omega).
\end{align*}
Again, due to orthogonality it follows that $\int \limits_\Omega  \partial_x f\Big(\sum \limits_{k=0}^N (\reconst)_k \Psi_k\Big) \Psi_l ~\mathrm{d}\p(\omega) =0$ 
for $l>2N$. We therefore only need to compute $(\fR^{stoch})_l$ for $l=N+1,\ldots,2N$. Moreover,
\begin{align*}
\frac{1}{2} \int \limits_\Omega  \partial_x \Big(\sum \limits_{k=0}^N (\reconst)_k \Psi_k\Big)^2 \Psi_l ~ \mathrm{d}\p(\omega)  &= (\partial_x \reconst)^\top C_l \reconst  ,
\end{align*}
where $[C_k]_{i,j=0}^N:= \int \limits_{\Omega} \Psi_j\Psi_i\Psi_k ~ \mathrm{d}\p(\omega)$, $k =N+1,\ldots,2N$, is the symmetric triple product matrix. 
\item For a multi-element approach, as discussed in Remark \ref{polys}, we can compute the spatial and stochastic residual on each stochastic element 
$I_l^{N_e}$, $l=0,\ldots, 2^{N_e}$ in parallel.
\end{enumerate}
\end{example}
\section{Numerical Examples}
In this section we present numerical experiments, where we examine the scaling behavior of the space-time-stochastic residual. 
For the following test cases, we consider the classical polynomial chaos expansion using Legendre orthonormal polynomials for uniformly distributed random variables. 
As numerical solver for the SG system we use the Runge-Kutta Discontinuous Galerkin solver \textit{Flexi} \cite{Hindenlang2012}. 
For the time-stepping we use an explicit RK3-7 method \cite{TOULORGE20122067}, a time-reconstruction of order three and for the physical space we use  DG polynomials of degree one or two.
As numerical flux for the linear advection equation, we choose the upwind flux, 
 \begin{align*}
 \G(\fu,\fv)= \ff(\fw(\fu,\fv)), ~\fw(\fu,\fv)= \fu,
\end{align*}
and as projection operator for the initial data, we choose the Radau-projection operator $\mathcal{I}_{V_p^s}= \mathbb{R}_h^{+}$, as defined in \cite{Radau}.
For the Burgers equation, where we have no upwind transport, we use the Lax-Wendroff numerical flux 
 \begin{align*}
 \G(\fu,\fv)= \ff(\fw(\fu,\fv)), ~\fw(\fu,\fv)= \frac{1}{2}\Big( ({\fu +\fv})+ \frac{\Delta t}{ h} \big(\ff(\fv)- \ff(\fu) \big) \Big).
\end{align*}
In this case, we use Gauss-Legendre interpolation of the initial data in our numerical experiments. 
We have also tried other interpolation and projection operators and all of them lead to the same scaling behavior of the space-time residual.
The integrals are numerically evaluated  by a Gau\ss-Legendre quadrature in space, time and the stochastic space. For the time integration we use 8 points, in the physical space we use 25 points, and in the stochastic space 80 points.
\begin{remark}
As already mentioned after Remark \ref{rem:ini}, for nonlinear hyperbolic equations, the space-time residual is suboptimal (by one order) on the first time step.
Indeed, we loose half an order of convergence in the (global) space-time residual, i.e. when the error is of order $h^{p+1}$ the error estimator is of order $h^{p+1/2}$. 
This is due to a lack of compatibility between the projection/interpolation of the initial data into the DG space and the spatial-reconstruction. 
For the linear advection equation, where we compute the numerical solution using an upwind numerical flux, the Radau projection is the compatible choice, 
as it accounts for the upwind direction. 
Indeed, we have used it in our numerical experiments and observed optimal rates for the error estimator. 
A similar concept for nonlinear equations is up to now missing.
We would like to stress once again that the  estimates based on Kru\v{z}kov's doubling of variables technique leads to a much more reduced rate of convergence for smooth solutions, namely $h^{(p+1)/2}$.
If we start to reconstruct the numerical solution of the Burgers equation from $t=0$, we loose half an order of convergence in the space-time residual. 
Therefore, we start to reconstruct the numerical solution after the first time step on the coarsest mesh, where we conducted our computations. This corresponds to $t=0.008$. 
We also start to integrate the space-time residual from $t=0.008$ and obtain the full order of convergence in the space-time residual. 
\end{remark}
\subsection{The linear advection equation} 
We consider the linear advection equation
$ \partial_t u +2 \partial_x u =0,$
on the spatial domain $[0,2]_{per}$ and with $T=0.2$. We start the computation with 16 elements and a time-step size of $\Delta t=0.02$. 
We then subsequently reduce $h$ and $\Delta t$ by a factor of two. The initial condition is given by 
$ u^0(x,\xi)= \xi(1-0.5\cos(\pi x)),$ where we assume $\xi \sim \mathcal{U}[1,3]$ to be uniformly distributed. 
In Figures \ref{advp0} and \ref{advp2} we show the space-stochastic numerical error between the exact solution $u(t,x,\xi)= \xi(1-0.5\cos(\pi (x-2t)))$ and the numerical solution 
computed by the Runge-Kutta Discontinuous Galerkin method for one and three chaos polynomials ($N=0, 2$), evaluated at $t_n=T$. 
Thanks to $C_{f''}=0$ in this special case, the exponential term of the error indicator from Corollary \ref{stochApost2} vanishes. 
We further plot the norm of the space-time residual, denoted by $\cR^{st}$, as in Theorem \ref{normOrtho}. 
In Figure  \ref{advp0} we can see that the space-stochastic error is not decreasing when $h$ tends to zero. 
This is due to the term $\cE^{stoch}_0$, cf. Remark \ref{rem:ini}. 
The overall error is dominated by the error we make in projecting the initial condition onto span$\{\Psi_0\}$. 
If we increase the number of orthonormal polynomials to three, we obtain an exact representation of the initial condition in the orthonormal basis, i.e., $\cE^{stoch}_0=0$. 
Therefore, the space-stochastic numerical error only consists of the space-time discretization error, this can be seen in Figure \ref{advp2}.  
After increasing the polynomial chaos degree, the space-stochastic error  converges with the same order as the spatial residual.
Furthermore, for $p=1,2$ both residuals have the correct order of convergence, that is $p+1$.
\begin{figure}[!h]
\centering
\includegraphics[width=\figurewidth, height=\figureheight]{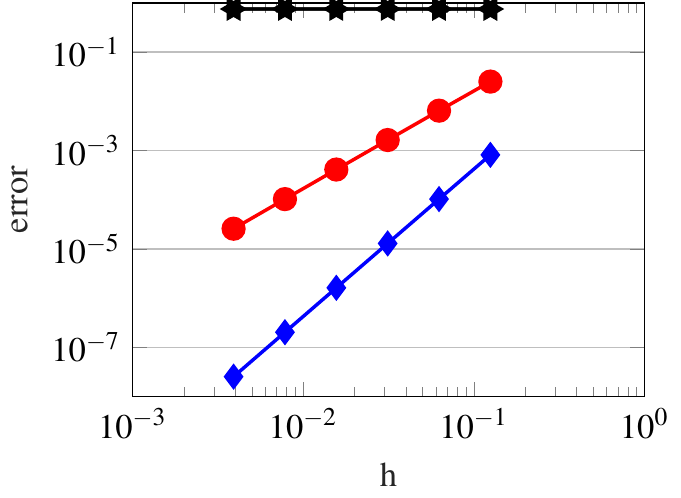}
\includegraphics[width=\figurewidth, height=1.2\figureheight]{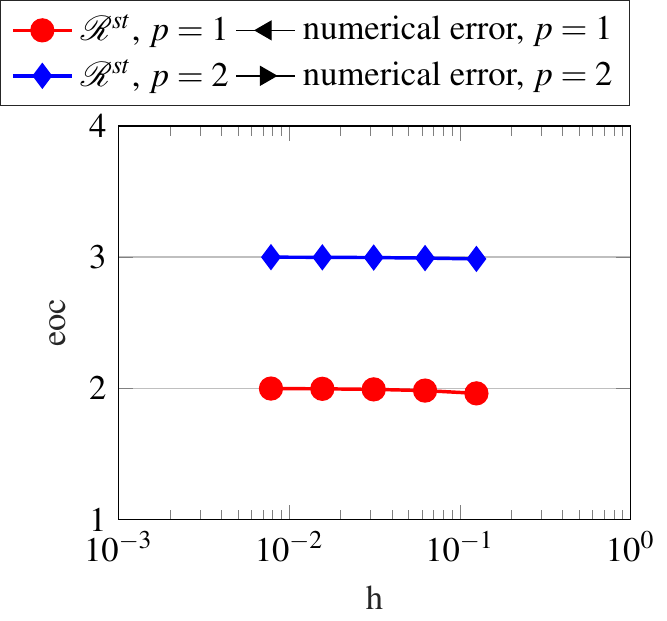}
\caption{Error and eoc plot for the linear advection equation in the case of one \newline orthonormal polynomial and DG polynomial degrees $p=1,2$.}\label{advp0}
\end{figure}
\begin{figure}[!h]
\centering
\includegraphics[width=\figurewidth, height=\figureheight]{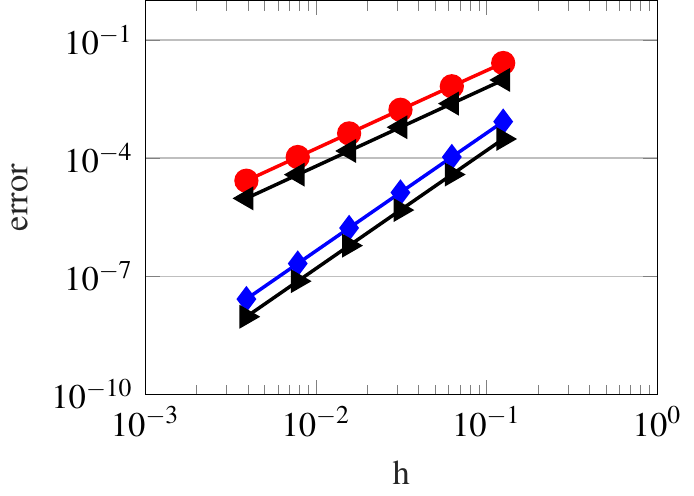}
\includegraphics[width=\figurewidth, height=1.2\figureheight]{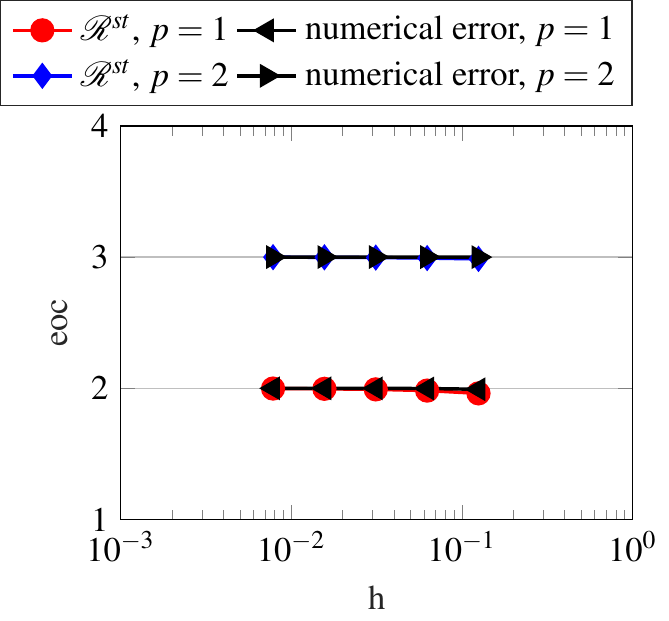}
\caption{Error and eoc plot for the linear advection equation in the case of three \newline orthonormal polynomials and DG polynomial degrees $p=1,2$.}\label{advp2}
\end{figure}
\subsection{The Burgers equation with smooth solution}
In this numerical example we consider Burgers' equation $\partial_t u+ \partial_x(\frac{u^2}{2})=S$ with the source term 
$$S(t,x,\xi)= \pi \xi^2 \sin(\pi (x-\xi t)) \big(\cos(\pi (x-\xi t)) -1 \big),$$ $\xi \sim \mathcal{U}[1,3]$.  For the initial condition $u^0(x,\xi)= \xi\cos(\pi x)$, the exact solution is 
$$u(t,x,\xi)= \xi\cos(\pi (x-\xi t)). $$
The numerical solution is computed up to time $T=0.2$ on the spatial domain $[0,2]_{per}$ and we start the computations initially on a mesh with 16 elements 
and time step size  $\Delta t =0.008$. Again we reduce $h$ and $\Delta t$ by a factor of two. The DG polynomial degree is two and the reconstruction in time is of order three. 
In the following numerical computations we consider different cases, where on the one hand we refine the physical space 
and on the other hand we increase the polynomial degree of the orthonormal (chaos) polynomials. 
The latter corresponds to a refinement in the stochastic space. 
We plot $\cR^{st}$, $\cR^{stoch}$  as in Theorem \ref{normOrtho} and $\| u(T,\cdot,\cdot) - u_h^{N_t} \|_{L^2((0,2) \times \Omega)}$,
which we call numerical error.
\begin{figure}[!htb] 
\centering
\subfigure[Error plot for the Burgers equation in  the case of  one  \newline orthonormal polynomial and DG polynomial degree $p=2$.]{
\includegraphics[width=\figurewidth, height=\figureheight]{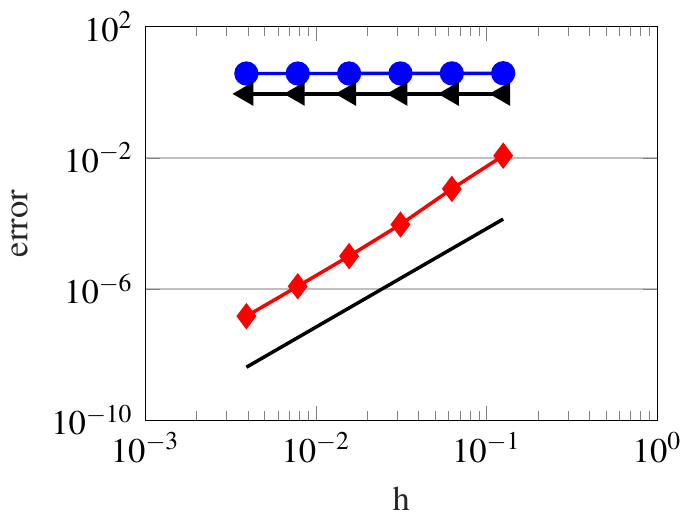} }
\subfigure[Error plot for the Burgers equation in  the case of  five  \newline orthonormal polynomials  and DG polynomial degree $p=2$.]{
\includegraphics[width=\figurewidth, height=\figureheight]{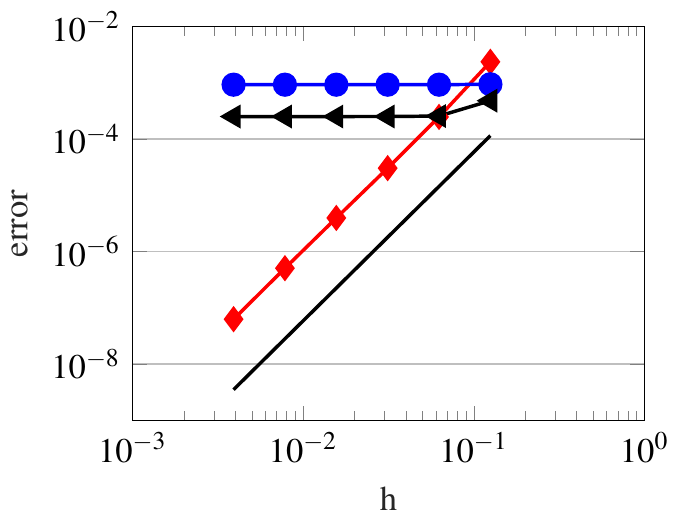} } 
\subfigure[Error plot for the Burgers equation in  the case of  thirteen  \newline orthonormal polynomials  and DG polynomial degree $p=2$.]{
\includegraphics[width=\figurewidth, height=1.2\figureheight]{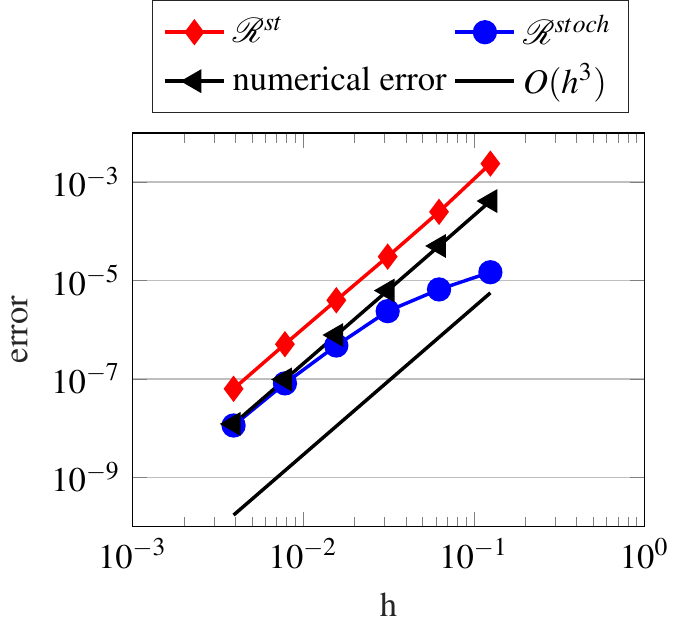}} 
\subfigure[Behavior of exponential factor for one,  five and thirteen \newline orthonormal polynomials.]{
\includegraphics[width=\figurewidth, height=1.2\figureheight]{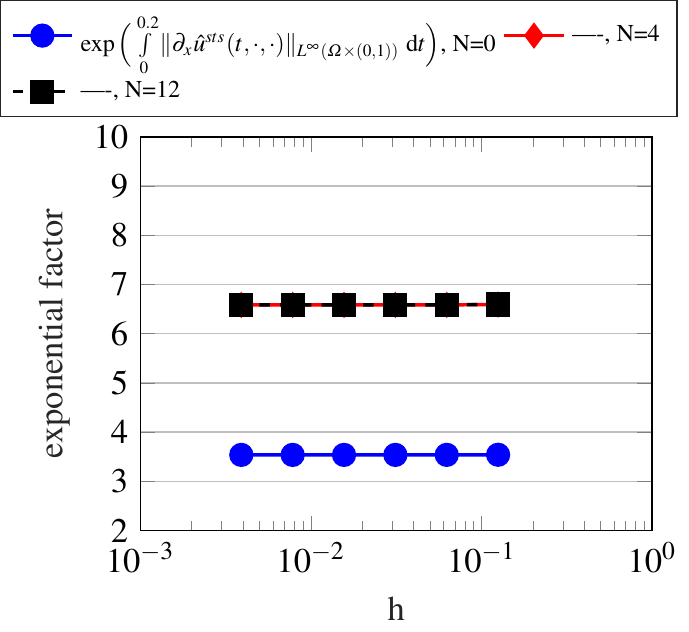} } 
\caption{Error plots for the Burgers equation for $h$-refinement. }
\label{plot:BurgershRef}
\end{figure}
We start our computations by considering mesh refinements in the physical space with a fixed polynomial chaos degree. 
In Figure  \ref{plot:BurgershRef} (a) we display the numerical solution using only one chaos polynomial, which corresponds to $N=0$. 
We can see that the overall error is clearly dominated by  $\cE_0^{stoch} + \cR^{stoch}$. To reduce the overall error significantly, 
we have to increase the polynomial chaos degree.
We increase the number of polynomials to five, corresponding to $N=4$ . We observe in Figure \ref{plot:BurgershRef} (b), 
that for the coarse space discretization with 16 elements the space-stochastic error is dominated by the space-time residual.  
After this point any significant reduction of the overall error requires again an increase of the polynomial chaos degree.
When increasing the polynomial degree to thirteen ($N=12$), we can see in Figure \ref{plot:BurgershRef} (c) that the overall error is now dominated by the space-time residual 
as the stochastic discretization is fine enough. The overall error now converges with the same rate as the space-time residual. 
Additionally, in Figure \ref{plot:BurgershRef} (d) we plot the exponential factor from Theorem \ref{stochApost} for different polynomial degrees $N$ and different mesh sizes $h$.  We can see that in the smooth case the exponential factor stays bounded for $h\to 0$. The exponential factor for $N=0$ is smaller than for $N=4,12$ because we solve the (SG)-system only for the mean value.

In the previous computations we have considered spatial refinements for a fixed polynomial chaos degree. Now we want to examine the behavior of  $\cR^{stoch}$ for different mesh sizes and a DG polynomial degree two. 
 In Figure \ref{plot:BurgersNRef} (a) we show results for a fixed spatial discretization with 16 elements.
We can see that the space-stochastic error is dominated by the space-time residual, because the spatial discretization is too coarse. 
We can also see, that the spatial residual remains unchanged by increasing the polynomial chaos degree. 
Additionally, we note that the stochastic residual exhibits spectral convergence.
 To reduce the overall error we therefore need to increase the number of spatial elements or the DG polynomial degree.
Finally, in Figure \ref{plot:BurgersNRef} (b) we consider a very fine mesh, consisting of 1024 elements. 
Due to the fine resolution of the physical space, the overall error is dominated by the stochastic residual up to $N=8$ and can only be decreased by increasing the polynomial chaos degree.
After that point, the overall error can only be significantly decreased by performing a spatial refinement. 
We can now also see how the overall space-stochastic error converges spectrally until its convergence is again dominated by the spatial error.

Let us note that the spectral convergence of the stochastic residual is in accordance with what is known theoretically for stochastic discretization errors in SG schemes.
Indeed, the authors of \cite{420bdedc3eb340c0b55bf59b120d7e5d} prove spectral convergence of the SG method for the linear transport equation with random transport velocity. 
This indicates that the stochastic residual allows us to assess the magnitude of the stochastic discretization error in an optimal (spectral) way.
\begin{figure}[!h]
\centering
\subfigure[Error plot for the Burgers equation for a fixed mesh with  16 elements and DG polynomial degree $p=2$.]{
\includegraphics[width=\figurewidth, height=\figureheight]{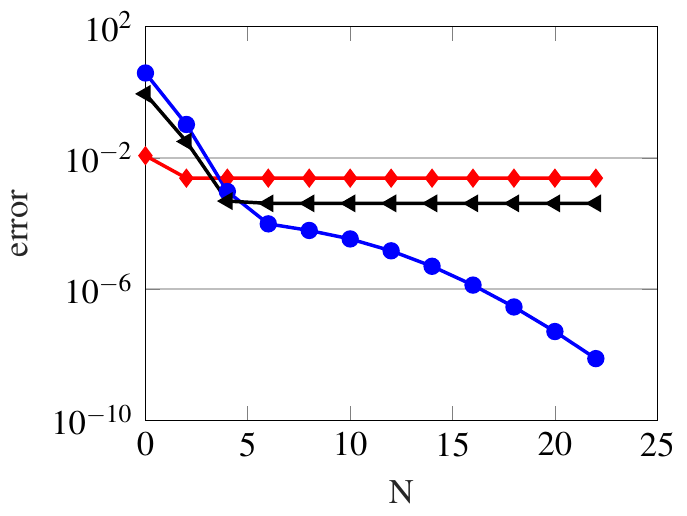} }
\subfigure[Error plot for the Burgers equation for a fixed mesh with  1024 elements and DG polynomial degree $p=2$.]{
\includegraphics[width=\figurewidth, height=\figureheight]{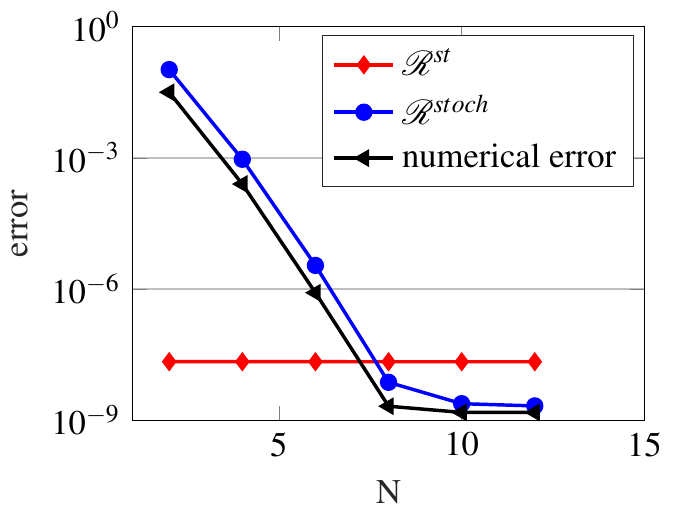}  } 
\caption{Error plots for the Burgers equation for $N$-refinement.}
\label{plot:BurgersNRef}
\end{figure}
\subsection{The Burgers equation, the artificial shock case}
\begin{figure}[!h]
\centering
\subfigure[$N=1$, residuals indicated on the right y-axis]{
\includegraphics[width=\figurewidth, height=\figureheight]{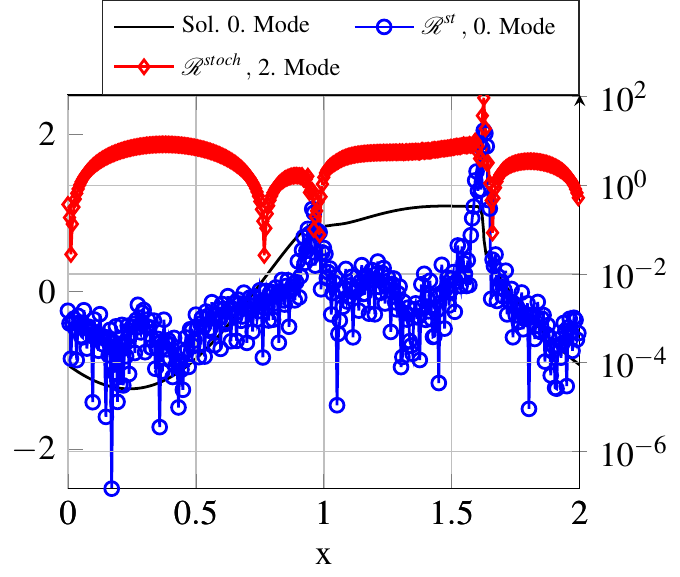} }
\subfigure[$N=2$, residuals indicated on the right y-axis] {
\includegraphics[width=\figurewidth, height=\figureheight]{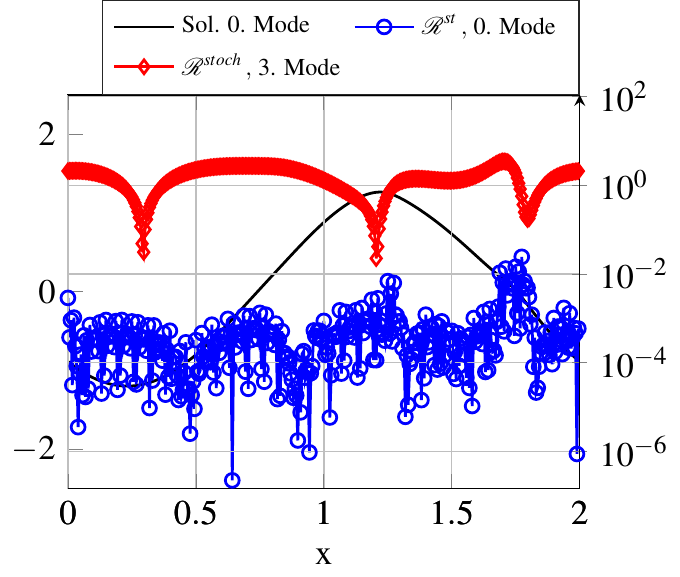} }
\caption{Plot of the numerical solution, spatial residual and stochastic residual for the Burgers \newline equation in the case of two and three orthonormal polynomial  and DG polynomial degree $p=2$.} 
\label{shock2}
\end{figure}
\begin{figure}[!h]
\centering
\includegraphics[width=2\figurewidth, height=1.2\figureheight]{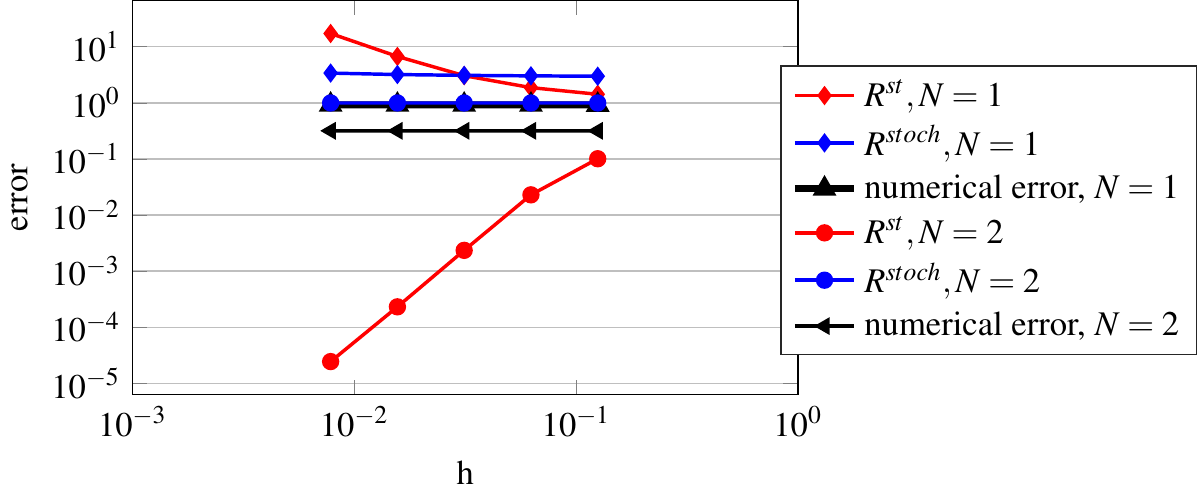}
\caption{Comparison of residuals and numerical error for one and 
two orthonormal polynomials and \newline DG polynomial degree $p=2$.} \label{shock1}
\end{figure}
We study now the same example as in the previous section, but we compute numerical solutions up to $T=0.56$ and use the slope limiter $\Lambda \Pi_h$ from \cite{CockburnShu2001}.
It is a well-known drawback of the SG-methodology, cf.~the numerical example in Section 6 of \cite{pettersson2015polynomial}, that, even if the solution to (RIVP)  is smooth up to time $T$, 
the solutions of the (\ref{SGSystem})-system may develop discontinuities before that time.
Indeed, this is the case in the example at hand.
In Figure \ref{shock2}(a), we display the numerical solutions and residuals for $N=1$ at time $T$.
The solution of the zeroth mode appears to contain a shock at approximately $x\approx 1.6$.
We see that the spatial 
discontinuity is clearly picked up by the spatial residual and, to some extent, also in the stochastic residual.
Moreover, we can see in Figure \ref{shock1}, that for $h \to 0$ the spatial residual blows up, although the numerical error stays constant.
We may therefore use the spatial residual as an indicator for spatial mesh refinements, which will be left for future research.
 
Also the stochastic residual grows with $h\to 0$, however very slow compared to the spatial residual.
Increasing the polynomial chaos degree to $N=2$, also increases the smoothness of the numerical solution, which can be seen in Figure \ref{shock2} (b).
In Figure \ref{shock1}, we can see that for $N=2$ the spatial residual decreases when $h$ tends to zero. 

The artificial generation of shocks is a general problem of the SG method,
however our residuals were able to detect the position of the shock correctly.
This demonstrates that our error estimator picks up on the ``artificial'' discontinuity and can help the user to determine a more suitable polynomial degree in which  no such discontinuities are present.
It should be noted, however, that the estimator cannot distinguish between discontinuities in the solution $u$ of \eqref{stochCL} and ``artificial'' discontinuities in the solution of the (\ref{SGSystem})-system.
Furthermore, there are cases in which the ``artificial''  discontinuities can not as easily be avoided as in the example at hand. 
In one example in \cite[Section 6]{pettersson2015polynomial} each (SG)-solution contains shocks and for increasing $N$ the number of shocks increases while their size decreases.
From the numerical point of view (for some fixed $h>0$) a solution with very small shocks cannot be distinguished from a smooth solution.
In such a case it is probably possible to obtain a sequence of approximate solutions for which the error estimator converges by choosing $h$ and $N$ such that they are properly related,
but a detailed discussion of this issue is beyond the scope of this work.
\subsection{The Burgers equation, the shock case}
As last numerical example we consider a Riemann problem for Burgers' equation, where the jump size of the discontinuity and thus the shock speed is random. The spatial domain is now $[-1,1]$ and we set $T=0.1$. The initial condition is given by
\begin{align*} u^0(x,\xi)=
\begin{cases}
1+ \xi, &\text{ if } x\leq 0, \\
0.5+ \xi &\text{ else,}
\end{cases}
\end{align*}
where $\xi \sim \mathcal{U}[-0.2,0.2]$. The shock speed can be computed as $s(\xi) = 1.5+2\xi$ and therefore the exact solution is given by
\begin{align*} u(t,x,\xi)=
\begin{cases}
1+ \xi, &\text{ if } x\leq s(\xi)t\\
0.5+ \xi &\text{ else.}
\end{cases}
\end{align*}
As the shock speed is $\P-$a.s.~positive, we use the upwind numerical flux in this numerical experiment. We use the slope limiter $\Lambda \Pi_h $ from \cite{CockburnShu2001} and consider a very fine physical resolution with 512 elements and DG polynomial degree of 2, Although the exact solution is discontinuous we can see in Figure \ref{fig:realshock} that the stochastic residual exhibits exponential convergence for a sufficiently large number of orthonormal polynomials. This is due to the fact that the coefficients in the polynomial chaos expansion are smooth, cf. \cite[Section 6]{pettersson2015polynomial}.
Hence, although the exact solution is discontinuous the stochastic residual displays the correct (spectral) type of convergence and moreover,
it gives us information about the resolution in the stochastic space independent of the spatial resolution.
\begin{figure}[!h]
\makebox[\linewidth]{ 
\includegraphics[width=2\figurewidth, height=1.2\figureheight]{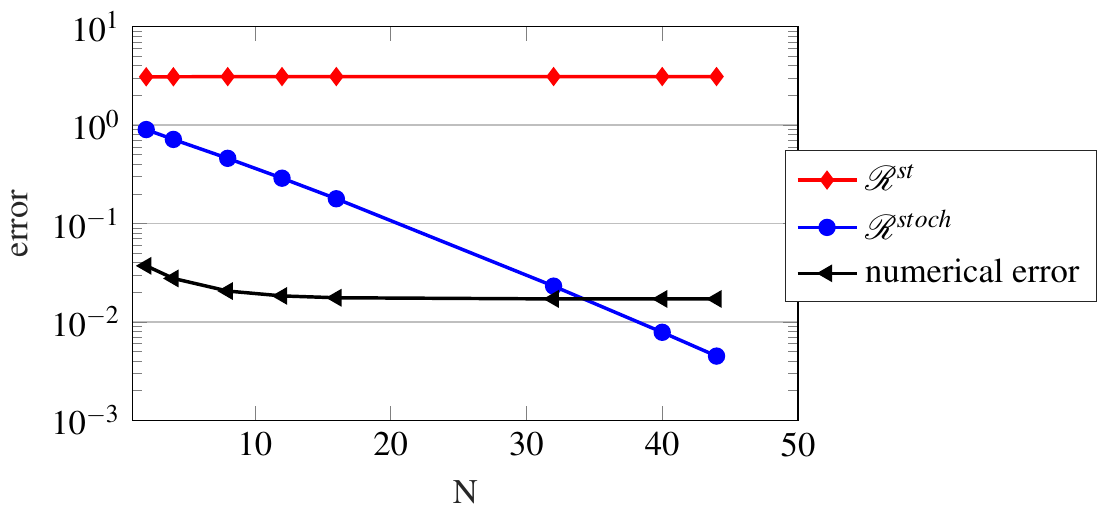} }
\caption{Error plot for the Burgers equation for a fixed mesh with 512 elements and DG \newline polynomial degree $p=2$.}
 \label{fig:realshock}
\end{figure}
\section{Conclusions}

This work provides a first rigorous a posteriori error analysis of scalar conservation laws with uncertain initial conditions and source terms.
We derived an a posteriori error estimator for the random conservation law and additionally showed that the space-time-stochastic residual admits an orthogonal splitting 
into a residual representing the space-time error and a residual, representing the stochastic error. 
Numerical experiments demonstrate that the two parts of the error estimator scale in $N$ and $h$ as expected. Moreover, in the case of the linear advection equation, the space-time residual exhibits the optimal order of convergence.
Furthermore, the numerical experiments showed that both residuals serve as indicators to determine if, after a completed computation, 
one should increase the polynomial chaos degree or perform a spatial refinement to decrease the overall numerical error. 
In the case of a spatial discontinuity, the discontinuity is also visible in the stochastic residual. 

Future work will focus on the a posteriori error analysis of systems of conservation laws with uncertain initial conditions or uncertain parameters, 
for example for the Euler equations. Due to the relative entropy framework we expect our theory to be extendable to this case. 
For further applications of our method, the construction of space-stochastic adaptive schemes using the Hybrid Stochastic Galerkin method  (\cite{MR3269040}) and the residuals as local indicators, will be considered.
\section*{Acknowledgements}

The authors thank the Baden-W{\"u}rttemberg Stiftung for support via the project ``Numerical Methods
for Multiphase Flows with Strongly Varying Mach Numbers'' (J.G.) and the project ``SEAL'' (F.M., C.R.)

\bibliographystyle{IMANUM-BIB}
\bibliography{mybib2}

\begin{thebibliography}{}

\bibitem[Abgrall \& Mishra(2017)Abgrall \& Mishra]{bgrll2017}
{\sc Abgrall, R. \& Mishra, S.} (2017)
\newblock {\em {{C}hapter 19 {\textendash} {U}ncertainty {Q}uantification for
  {H}yperbolic {S}ystems of {C}onservation {L}aws}\/}.
\newblock Amsterdam: Elsevier.

\bibitem[Bijl {\em et~al.}(2013)Bijl, Lucor, Mishra, \& Schwab]{mishraUQ}
{\sc Bijl, H., Lucor, D., Mishra, S. \& Schwab, C.} (2013)
\newblock {\em Uncertainty quantification in computational fluid dynamics\/}.
  Lect. Notes Comput. Sci. Eng.,  vol.~92.
\newblock Springer, Heidelberg, pp. xii+333.

\bibitem[B\"urger {\em et~al.}(2014)B\"urger, Kr\"oker, \& Rohde]{MR3269040}
{\sc B\"urger, R., Kr\"oker, I. \& Rohde, C.} (2014)
\newblock A hybrid stochastic {G}alerkin method for uncertainty quantification
  applied to a conservation law modelling a clarifier-thickener unit.
\newblock {\em ZAMM Z. Angew. Math. Mech.}, {\bf 94}, 793--817.

\bibitem[Butler {\em et~al.}(2011)Butler, Dawson, \& Wildey]{MR2813239}
{\sc Butler, T., Dawson, C. \& Wildey, T.} (2011)
\newblock A posteriori error analysis of stochastic differential equations
  using polynomial chaos expansions.
\newblock {\em SIAM J. Sci. Comput.}, {\bf 33}, 1267--1291.

\bibitem[Cockburn(2003)Cockburn]{cockburn2003continuous}
{\sc Cockburn, B.} (2003)
\newblock Continuous dependence and error estimation for viscosity methods.
\newblock {\em Acta Numer.}, {\bf 12}, 127--180.

\bibitem[Cockburn \& Shu(1998)Cockburn \& Shu]{cockburn1998runge}
{\sc Cockburn, B. \& Shu, C.-W.} (1998)
\newblock The {R}unge-{K}utta discontinuous {G}alerkin method for conservation
  laws. {V}. {M}ultidimensional systems.
\newblock {\em J. Comput. Phys.}, {\bf 141}, 199--224.

\bibitem[Cockburn \& Shu(2001)Cockburn \& Shu]{CockburnShu2001}
{\sc Cockburn, B. \& Shu, C.-W.} (2001)
\newblock Runge-{K}utta discontinuous {G}alerkin methods for
  convection-dominated problems.
\newblock {\em J. Sci. Comput.}, {\bf 16}, 173--261.

\bibitem[Dafermos(2016)Dafermos]{dafermos2005hyperbolic}
{\sc Dafermos, C.~M.} (2016)
\newblock {\em Hyperbolic conservation laws in continuum physics\/}.
  Grundlehren der Mathematischen Wissenschaften [Fundamental Principles of
  Mathematical Sciences],  vol. 325, fourth edn.
\newblock Springer-Verlag, Berlin, pp. xxxviii+826.

\bibitem[Deb {\em et~al.}(2001)Deb, Babu{\v s}ka, \& Oden]{MR1870425}
{\sc Deb, M.~K., Babu{\v s}ka, I.~M. \& Oden, J.~T.} (2001)
\newblock Solution of stochastic partial differential equations using
  {G}alerkin finite element techniques.
\newblock {\em Comput. Methods Appl. Mech. Engrg.}, {\bf 190}, 6359--6372.

\bibitem[Dedner {\em et~al.}(2007)Dedner, Makridakis, \&
  Ohlberger]{dedner2007error}
{\sc Dedner, A., Makridakis, C. \& Ohlberger, M.} (2007)
\newblock Error control for a class of {R}unge-{K}utta discontinuous {G}alerkin
  methods for nonlinear conservation laws.
\newblock {\em SIAM J. Numer. Anal.}, {\bf 45}, 514--538.

\bibitem[Dedner \& Giesselmann(2016)Dedner \& Giesselmann]{DG_16}
{\sc Dedner, A. \& Giesselmann, J.} (2016)
\newblock A posteriori analysis of fully discrete method of lines discontinuous
  {G}alerkin schemes for systems of conservation laws.
\newblock {\em SIAM J. Numer. Anal.}, {\bf 54}, 3523--3549.

\bibitem[Despr\'es {\em et~al.}(2013)Despr\'es, Po\"ette, \&
  Lucor]{Despres2013}
{\sc Despr\'es, B., Po\"ette, G. \& Lucor, D.} (2013)
\newblock Robust uncertainty propagation in systems of conservation laws with
  the entropy closure method.
\newblock {\em Uncertainty quantification in computational fluid dynamics\/}.
  Lect. Notes Comput. Sci. Eng., vol. 92.
\newblock Springer, Heidelberg, pp. 105--149.

\bibitem[Eigel {\em et~al.}(2014)Eigel, Gittelson, Schwab, \&
  Zander]{MR3154028}
{\sc Eigel, M., Gittelson, C.~J., Schwab, C. \& Zander, E.} (2014)
\newblock Adaptive stochastic {G}alerkin {FEM}.
\newblock {\em Comput. Methods Appl. Mech. Engrg.}, {\bf 270}, 247--269.

\bibitem[Erd\'elyi {\em et~al.}(1981)Erd\'elyi, Magnus, Oberhettinger, \&
  Tricomi]{MR698780}
{\sc Erd\'elyi, A., Magnus, W., Oberhettinger, F. \& Tricomi, F.~G.} (1981)
\newblock {\em Higher transcendental functions. {V}ol. {II}\/}.
\newblock Robert E. Krieger Publishing Co., Inc., Melbourne, Fla., pp.
  xviii+396.
\newblock Based on notes left by Harry Bateman, Reprint of the 1953 original.

\bibitem[Georgoulis {\em et~al.}(2014)Georgoulis, Hall, \&
  Makridakis]{Georgoulis2014}
{\sc Georgoulis, E.~H., Hall, E. \& Makridakis, C.} (2014)
\newblock Error control for discontinuous {G}alerkin methods for first order
  hyperbolic problems.
\newblock {\em Recent developments in discontinuous {G}alerkin finite element
  methods for partial differential equations\/}. IMA Vol. Math. Appl., vol.
  157.
\newblock Springer, Cham, pp. 195--207.

\bibitem[Ghanem \& Spanos(1991)Ghanem \& Spanos]{Ghanem}
{\sc Ghanem, R.~G. \& Spanos, P.~D.} (1991)
\newblock {\em Stochastic finite elements: a spectral approach\/}.
\newblock Springer-Verlag, New York, pp. x+214.

\bibitem[Giesselmann {\em et~al.}(2015)Giesselmann, Makridakis, \&
  Pryer]{GMP_15}
{\sc Giesselmann, J., Makridakis, C. \& Pryer, T.} (2015)
\newblock A posteriori analysis of discontinuous {G}alerkin schemes for systems
  of hyperbolic conservation laws.
\newblock {\em SIAM J. Numer. Anal.}, {\bf 53}, 1280--1303.

\bibitem[Gosse \& Makridakis(2000)Gosse \& Makridakis]{gosse2000two}
{\sc Gosse, L. \& Makridakis, C.} (2000)
\newblock Two a posteriori error estimates for one-dimensional scalar
  conservation laws.
\newblock {\em SIAM J. Numer. Anal.}, {\bf 38}, 964--988.

\bibitem[Gottlieb \& Xiu(2008)Gottlieb \&
  Xiu]{420bdedc3eb340c0b55bf59b120d7e5d}
{\sc Gottlieb, D. \& Xiu, D.} (2008)
\newblock Galerkin method for wave equations with uncertain coefficients.
\newblock {\em Commun. Comput. Phys.}, {\bf 3}, 505--518.

\bibitem[Hindenlang {\em et~al.}(2012)Hindenlang, Gassner, Altmann, Beck,
  Staudenmaier, \& Munz]{Hindenlang2012}
{\sc Hindenlang, F., Gassner, G.~J., Altmann, C., Beck, A., Staudenmaier, M. \&
  Munz, C.-D.} (2012)
\newblock Explicit discontinuous {G}alerkin methods for unsteady problems.
\newblock {\em Comput. \& Fluids\/}, {\bf 61}, 86--93.

\bibitem[Hoang \& Schwab(2013)Hoang \& Schwab]{20.500.11850/154978}
{\sc Hoang, V.~H. \& Schwab, C.} (2013)
\newblock Sparse tensor {G}alerkin discretization of parametric and random
  parabolic {PDE}s---analytic regularity and generalized polynomial chaos
  approximation.
\newblock {\em SIAM J. Math. Anal.}, {\bf 45}, 3050--3083.

\bibitem[Hu {\em et~al.}(2015)Hu, Jin, \& Xiu]{Xiu}
{\sc Hu, J., Jin, S. \& Xiu, D.} (2015)
\newblock A stochastic {G}alerkin method for {H}amilton-{J}acobi equations with
  uncertainty.
\newblock {\em SIAM J. Sci. Comput.}, {\bf 37}, A2246--A2269.

\bibitem[Jin \& Ma(2017)Jin \& Ma]{jin2017discrete}
{\sc Jin, S. \& Ma, Z.} (2017)
\newblock {The discrete Stochastic Galerkin method for hyperbolic equations
  with non-smooth and random coefficients}.
\newblock {\em J. Sci. Comp\/}.

\bibitem[K\"oppel {\em et~al.}(2017)K\"oppel, Kr\"oker, \& Rohde]{MR3671657}
{\sc K\"oppel, M., Kr\"oker, I. \& Rohde, C.} (2017)
\newblock Intrusive uncertainty quantification for hyperbolic-elliptic systems
  governing two-phase flow in heterogeneous porous media.
\newblock {\em Comput. Geosci.}, {\bf 21}, 807--832.

\bibitem[Kr\"oker {\em et~al.}(2015)Kr\"oker, Nowak, \& Rohde]{MR3351779}
{\sc Kr\"oker, I., Nowak, W. \& Rohde, C.} (2015)
\newblock A stochastically and spatially adaptive parallel scheme for uncertain
  and nonlinear two-phase flow problems.
\newblock {\em Comput. Geosci.}, {\bf 19}, 269--284.

\bibitem[Kr\"oner \& Ohlberger(2000)Kr\"oner \&
  Ohlberger]{kroner2000posteriori}
{\sc Kr\"oner, D. \& Ohlberger, M.} (2000)
\newblock A posteriori error estimates for upwind finite volume schemes for
  nonlinear conservation laws in multidimensions.
\newblock {\em Math. Comp.}, {\bf 69}, 25--39.

\bibitem[Kru{\v z}kov(1970)Kru{\v z}kov]{kruvzkov1970first}
{\sc Kru{\v z}kov, S.~N.} (1970)
\newblock First order quasilinear equations with several independent variables.
\newblock {\em Mat. Sb. (N.S.)\/}, {\bf 81 (123)}, 228--255.

\bibitem[Le~Ma{\^i}tre {\em et~al.}(2004)Le~Ma{\^i}tre, Knio, Najm, \&
  Ghanem]{LEMAITRE200428}
{\sc Le~Ma{\^i}tre, O.~P., Knio, O.~M., Najm, H.~N. \& Ghanem, R.~G.} (2004)
\newblock Uncertainty propagation using {W}iener-{H}aar expansions.
\newblock {\em J. Comput. Phys.}, {\bf 197}, 28--57.

\bibitem[Le~Ma{\^ i}tre \& Knio(2010)Le~Ma{\^ i}tre \&
  Knio]{maitre2010spectral}
{\sc Le~Ma{\^ i}tre, O.~P. \& Knio, O.~M.} (2010)
\newblock {\em Spectral methods for uncertainty quantification: With
  applications to computational fluid dynamics\/}.
\newblock Scientific Computation.
\newblock Springer, New York, pp. xvi+536.

\bibitem[Mathelin \& Le~Ma{\^\i}tre(2007)Mathelin \& Le~Ma{\^\i}tre]{MR2369381}
{\sc Mathelin, L. \& Le~Ma{\^\i}tre, O.~P.} (2007)
\newblock Dual-based a posteriori error estimate for stochastic finite element
  methods.
\newblock {\em Commun. Appl. Math. Comput. Sci.}, {\bf 2}, 83--115.

\bibitem[Mishra \& Schwab(2012)Mishra \& Schwab]{Mishra}
{\sc Mishra, S. \& Schwab, C.} (2012)
\newblock Sparse tensor multi-level {M}onte {C}arlo finite volume methods for
  hyperbolic conservation laws with random initial data.
\newblock {\em Math. Comp.}, {\bf 81}, 1979--2018.

\bibitem[M\"uller(2003)M\"uller]{MR1952371}
{\sc M\"uller, S.} (2003)
\newblock {\em Adaptive multiscale schemes for conservation laws\/}. Lecture
  Notes in Computational Science and Engineering,  vol.~27.
\newblock Springer-Verlag, Berlin, pp. xiv+181.

\bibitem[Pettersson {\em et~al.}(2015)Pettersson, Iaccarino, \&
  Nordstr\"om]{pettersson2015polynomial}
{\sc Pettersson, M.~P., Iaccarino, G. \& Nordstr\"om, J.} (2015)
\newblock {\em Polynomial chaos methods for hyperbolic partial differential
  equations: Numerical techniques for fluid dynamics problems in the presence
  of uncertainties\/}.
\newblock Mathematical Engineering.
\newblock Springer, Cham, pp. xii+214.

\bibitem[Pultarov\'a(2015)Pultarov\'a]{MR3396480}
{\sc Pultarov\'a, I.} (2015)
\newblock Adaptive algorithm for stochastic {G}alerkin method.
\newblock {\em Appl. Math.}, {\bf 60}, 551--571.

\bibitem[Toulorge \& Desmet(2012)Toulorge \& Desmet]{TOULORGE20122067}
{\sc Toulorge, T. \& Desmet, W.} (2012)
\newblock Optimal {R}unge-{K}utta schemes for discontinuous {G}alerkin space
  discretizations applied to wave propagation problems.
\newblock {\em J. Comput. Phys.}, {\bf 231}, 2067--2091.

\bibitem[Tryoen {\em et~al.}(2012)Tryoen, Le~Ma{\^ i}tre, \&
  Ern]{doi:10.1137/120863927}
{\sc Tryoen, J., Le~Ma{\^ i}tre, O.~P. \& Ern, A.} (2012)
\newblock Adaptive anisotropic spectral stochastic methods for uncertain scalar
  conservation laws.
\newblock {\em SIAM J. Sci. Comput.}, {\bf 34}, A2459--A2481.

\bibitem[Wan \& Karniadakis(2006)Wan \& Karniadakis]{doi:10.1137/050627630}
{\sc Wan, X. \& Karniadakis, G.~E.} (2006)
\newblock Multi-element generalized polynomial chaos for arbitrary probability
  measures.
\newblock {\em SIAM J. Sci. Comput.}, {\bf 28}, 901--928.

\bibitem[Wiener(1938)Wiener]{Wiener}
{\sc Wiener, N.} (1938)
\newblock The {H}omogeneous {C}haos.
\newblock {\em Amer. J. Math.}, {\bf 60}, 897--936.

\bibitem[Xiu \& Karniadakis(2002)Xiu \& Karniadakis]{XiuKarniadakis}
{\sc Xiu, D. \& Karniadakis, G.~E.} (2002)
\newblock The {W}iener-{A}skey polynomial chaos for stochastic differential
  equations.
\newblock {\em SIAM J. Sci. Comput.}, {\bf 24}, 619--644.

\bibitem[Yan(2005)Yan]{10.2307/4101387}
{\sc Yan, Y.} (2005)
\newblock Galerkin finite element methods for stochastic parabolic partial
  differential equations.
\newblock {\em SIAM J. Numer. Anal.}, {\bf 43}, 1363--1384.

\bibitem[Zhang \& Shu(2010)Zhang \& Shu]{Radau}
{\sc Zhang, Q. \& Shu, C.-W.} (2010)
\newblock Stability analysis and a priori error estimates of the third order
  explicit {R}unge-{K}utta discontinuous {G}alerkin method for scalar
  conservation laws.
\newblock {\em SIAM J. Numer. Anal.}, {\bf 48}, 1038--1063.

\bibitem[Zhou \& Tang(2012)Zhou \& Tang]{EPFL-ARTICLE-185741}
{\sc Zhou, T. \& Tang, T.} (2012)
\newblock Convergence analysis for spectral approximation to a scalar transport
  equation with a random wave speed.
\newblock {\em J. Comput. Math.}, {\bf 30}, 643--656.

\end{thebibliography}

\clearpage
\end{document}